\documentclass[12pt]{amsart}

\usepackage{amsfonts}
\usepackage{amscd}
\usepackage{amssymb}
\usepackage{enumerate}
\allowdisplaybreaks

\usepackage{hyperref}

\newtheorem{thm}{Theorem}[section]
\newtheorem{cor}[thm]{Corollary}
\newtheorem{lem}[thm]{Lemma}

\newtheorem{prop}[thm]{Proposition}
\theoremstyle{definition}

\theoremstyle{remark}
\newtheorem{rem}[thm]{Remark}
\numberwithin{equation}{section}

\newcommand{\R}{\mathbb R}

\newcommand{\Na}{\mathbb N}

\newcommand{\C}{{\mathbb C}}

\newcommand{\pa}{\partial }

\renewcommand{\H}{\mathbb H}

\newcommand{\N}{\nabla }

\newcommand{\be}{\begin{equation}}
\newcommand{\ee}{\end{equation}}
\renewcommand{\Re}{\operatorname{Re}}
\renewcommand{\Im}{\operatorname{Im}}

\title[ Extension problem and trace Hardy inequality]
{An extension problem and trace Hardy inequality\\ for the sublaplacian on $H$-type groups}

\author[L. Roncal and S. Thangavelu]{Luz Roncal and Sundaram Thangavelu}

\address[L. Roncal]{BCAM - Basque Center for Applied Mathematics \\
Alameda de Mazarredo 14,
48009 Bilbao, Spain}
\email{lroncal@bcamath.org}

\address[S. Thangavelu]{Department of Mathematics\\
 Indian Institute of Science\\
560 012 Bangalore, India}
\email{veluma@math.iisc.ernet.in}


\thanks{L. R. was supported by grant
MTM2015-65888-C4-4-P
from Spanish Government, by the Basque Government through the
BERC 2014--2017 program, by Spanish Ministry of Economy and Competitiveness MINECO:
BCAM Severo Ochoa excellence accreditation SEV-2013-0323 and by Leonardo Foundation grant for Investigadores y Creadores Culturales 2017 from BBVA. S. T. was supported by J. C. Bose Fellowship from D. S. T., Government of India}

\begin{document}

\maketitle
\begin{abstract}
In this paper we study the extension problem for the sublaplacian on a $H$-type group and use the solutions to prove trace Hardy and Hardy inequalities for fractional powers of the sublaplacian.
\end{abstract}

\section{Introduction and main results}

Ever since Caffarelli and Silvestre \cite{Caffarelli-Silvestre} studied the extension problem associated to the Laplacian on $\R^n$ and realised the fractional power $(-\Delta)^{s/2}$ as the map taking Dirichlet data to the Neumann data, there has been a flurry of activities related to the extension problem. Fractional powers of Laplacians also occur naturally in conformal geometry and scattering theory. Indeed, as shown in the work of Chang-Gonz\'alez \cite{CG}, the fractional order Paneitz operators $P_{\gamma}$ arising in the work of Graham and Zworski \cite{GZ} in conformal geometry coincide with $(-\Delta)^{\gamma}$ when the conformally compact Einstein manifold is taken to be the hyperbolic space. In \cite{CG}, Chang-Gonz\'alez have also extended the definition of $(-\Delta)^{\gamma}$ for $\gamma\in (0,n/2)$. The main idea used in \cite{CG} is to relate the extension problem for the Laplacian on $\R^n$ to the scattering theory for the Laplace--Beltrami operator on the hyperbolic space $X=\R^{n+1}_+$ endowed with the hyperbolic metric $g_X=\frac{dy^2+d|x|^2}{y^2}$, $y>0$, $x\in \R^n$.

Not very long after the appearance of the work of Chang and Gonz\'alez, Frank et al \cite{FGMT} have studied the extension problem associated to the sublaplacian $\mathcal{L}$ on the Heisenberg group $\H^n$. Due to the fact that CR manifolds serve well as abstract models of real submanifolds of complex manifolds, there is a vast literature on CR geometry and analysis on CR manifolds. The role played by $\R^n$ in the case of conformal geometry is now played by the Heisenberg group. In this connection it has been observed that conformally invariant fractional powers of the sublaplacian, denoted by $\mathcal{L}_s$, are more relevant than the pure fractional powers $\mathcal{L}^s$, see \cite{BFM, FL}. In their work, Frank et al \cite{FGMT} have studied construction of CR covariant operators of fractional order on the Heisenberg group $\H^n$ and investigated how they may be constructed as the Dirichlet-to-Neumann operator associated to a degenerate elliptic equation in the spirit of Caffarelli-Silvestre.

The extension problem for the sublaplacian on $\H^n$ takes the form
\begin{equation}
\label{eq:epH}
\big(\partial_{\rho}^2+\frac{1-2s}{\rho}\partial_{\rho}+\frac14\rho^2\partial_t^2-\mathcal{L}\big)U=0 \qquad \text{ in } \H^n\times \R^+,
\end{equation}
with boundary condition $U(z,t,0)=f(z,t)$, $(z,t)\in \H^n$. Note that the extension problem is different from the usual problem due to the appearance of the extra term $\frac14\rho^2\partial_t^2$. When this extra term is absent, one can study the extension problem using the semigroup approach developed by Stinga-Torrea \cite{ST}, see also \cite{GMS}. However, if we consider $\H^n$ as the boundary of the Siegel's upper half space $\Omega_{n+1}$, then the above extension problem occurs naturally. Using this connection and making use of Fourier analysis on the Heisenberg group $\H^n$, Frank et al have shown that for $f\in C_0^{\infty}(\H^n)$ there is a unique solution of the above equation which satisfies 
$$
\mathcal{L}_sf=c_s\lim_{\rho\to0} \rho^{1-2s}\partial_{\rho}U.
$$
They have also proved an interesting trace inequality for the restriction map $T$ which takes functions $U$ on $\Omega_{n+1}$ into their boundary values on $\H^n$. In establishing their results, they have made use of results from scattering theory. 

In a recent article, M\"ollers et al \cite{MOZ} have looked at the extension problem associated to $\Delta$ on $\R^n$ and $\mathcal{L}$ on $\H^n$ in the light of representation theory. Realising that there are Lie groups of symmetries acting on the space of solutions of these boundary problems, they have related the solution operators (taking the boundary value into the solution) with symmetry-breaking operators constructed in the work of Kobayashi-Speh \cite{KS} and M\"ollers, \O rsted and Oshima \cite{MOO}. The Lie group relevant to the case of the extension problem for the sublaplacian $\mathcal{L}$ is the rank one real reductive group $U(1,n+2)$ and the solution operators which they call the Poisson transforms and denote by $P_s$, are related to the complementary series representations of this group.

M\"ollers et al \cite{MOZ} have taken the point of view that solutions $U(z,t,\rho)$ of the extension problem \eqref{eq:epH} can be considered as functions on the higher dimensional group $\H^{n+1}$ which are radial in the extra variable. Using coordinates $(z,\zeta,t)\in \C^n\times \C\times \R$ on $\H^{n+1}$ they have considered the operator
\begin{equation}
\label{eq:Ls}
L_s:=-|\zeta|^2\mathcal{L}+(1-2s)\big(\xi\frac{\partial}{\partial \xi}+\eta\frac{\partial}{\partial \eta}\big)
\end{equation}
where
$\zeta=\xi+i\eta$ and $\mathcal{L}$ is the sublaplacian on $\H^{n+1}$. The action of this operator on functions which are radial in the $\zeta$ variable leads to the operator
$$
-\mathcal{L}+\partial_{\rho}^2+\frac{1-2s}{\rho}\partial_{\rho}+\frac14\rho^2\partial_t^2
$$
which is related to the operator studied by Frank et al in \cite{FGMT}. The boundary value problem associated to the above operator can be solved explicitly giving the solution as a convolution of the boundary condition with a kernel known as the Poisson kernel. The complementary series representations $\pi_{-s}$ are realised on certain Hilbert spaces $H^s(\H^n)$ defined in terms of the fractional powers $\mathcal{L}_s$. Using representation theoretic arguments, M\"ollers et al \cite{MOZ} show that the solution operator $P_s$ is an isometry between $H^s(\H^n)$ and $H^{s+1}(\H^{n+1})$. 

In this article we revisit the extension problem in the more general context of $H$-type groups. Using techniques different from those used by Frank et al in \cite{FGMT} and M\"ollers et al in \cite{MOZ} we calculate explicitly the kernel associated to the solution operator $P_s$. This allows us to prove that the Dirichlet-to-Neumann map can be defined not only for $0<s<n+1$, as studied by Frank et al, but also for all $s>0$ save for a discrete set of forbidden values. By making use of the connection between the extension problem and the eigenfunction equation for the Laplace--Beltrami operator associated to the solvable extension $S$ of the given $H$-type group $N$, we characterise all solutions of the extension problem satisfiying uniform $L^p$ integrability, see Theorem \ref{thm:cha}. Moreover, by making use of the extension problem, we prove a trace Hardy inequality for the sublaplacian, see Theorem \ref{thm:GtHi}. Such a technique was already used in the Euclidean context in \cite{FMT, FMT2, Nguyen,Tz}. The trace Hardy inequality leads to a Hardy inequality with homogeneous weight function, see Corollary \ref{thm:Hardyhom}, which turns out to be sharp.

Let $ N $ be a $H$-type group whose  Lie algebra $ \mathfrak{n} $  is the direct sum of $ \mathfrak{v}$ and $ \mathfrak{z}.$ Here $ \mathfrak{z} $ is at the centre of the Lie algebra $\mathfrak{n}$ and $ \mathfrak{v} $ is even dimensional. Let $2n$ and $m$ denote the dimensions of $\mathfrak{v}$ and $\mathfrak{z}$ respectively. Let $ X_j$, $j = 1,2,...,2n $  and $ Z_j$, $j =1,2,...,m $ be bases for $ \mathfrak{v} $ and $ \mathfrak{z} $ respectively. The sublaplacian $ \mathcal{L} $ on $ N $ is defined by $ \mathcal{L} = -\sum_{j=1}^{2n} X_j^2 $, which is known to be a subelliptic operator. We are concerned with the extension problem
\begin{equation}
\label{eq:ep}
 \big( -\mathcal{L}+\partial_\rho^2 +\frac{1-2s}{\rho}\partial_\rho +\frac{1}{4}\rho^2 \Delta_z \big)u(v,z,\rho) = 0,  \qquad \lim_{\rho \rightarrow 0} u(v,z,\rho) = f(v,z)
\end{equation}
where $ \Delta_z = \sum_{k=1}^m Z_k^2 $ stands for the Laplacian on $ \R^m$.

We study solutions of this equation by relating them to eigenfunctions of the Laplace--Beltrami operator $\Delta_S$ on the solvable extension $S$ of the $H$-type group $N$. Recall that the $H$-type group $N$ admits nonisotropic dilations. Thus there is an action of $A=\R^+$ on $N$. We can therefore form the semidirect product of $N$ and $A$ which is usually denoted by $S=AN$. This group $S$ is solvable and when $N$ is an Iwasawa group coming out of a semisimple Lie group, $S$ can be identified with a noncompact Riemannian symmetric space of rank one. 

A basis for the Lie algebra $ \mathfrak{s} $ of $S$ is given by $ E_j = \sqrt{\rho}X_j$, $j=1,2,...,2n$, $T_k = \rho Z_k$, $k =1,2,...,m $ and $ H = \rho \partial_\rho$. The Laplace--Beltrami operator $\Delta_S$ on $S$ is defined by  
\begin{equation}  
\label{eq:LBel}
\Delta_S =  \sum_{j=1}^{2n} E_j^2 +\sum_{k=1}^m T_k^2 +H^2 -\frac{1}{2}Q H
\end{equation} 
where $ Q = 2(n+m) $ is the homogeneous dimension of $ N$. It can be shown that when $u$ satisfies the extension problem \eqref{eq:ep}, the function 
$$
\widetilde{w}(v,z,\rho)=\rho^{\frac{n+m-s}{2}}w(v,z,\rho)=\rho^{\frac{n+m-s}{2}}u(2^{-1/2}v,2^{-1}z,\sqrt{2z})
$$
satisfies the eigenfunction equation
$$
-\Delta_{S}\widetilde{w}(v,z,\rho)=\frac{(n+m)^2-s^2}{4}\widetilde{w}(v,z,\rho).
$$
Using this connection and known results on characterisations of eigenfunctions of the Laplace--Beltrami operator on $S$ we can prove the following theorem. 

Let 
\begin{equation} 
\label{eq:varfisr}
\varphi_{s,\rho}(v,z)  = \big((\rho^2+|v|^2)^2 +16 |z|^2\big)^{-\frac{n+m+s}{2}},
\end{equation} 
which is integrable on $N$ for all $s>0$. Let us also denote by $\ast$ the convolution on the group $N$. Thus for two functions $ f $ and $ g $ on $ N $ the convolution $ f \ast g $ is defined by
$$ f \ast g(x) = \int_{N} f(xy^{-1}) g(y) dy $$
where $ dy $ stands for the Haar measure on $ N$. 

\begin{thm}
\label{thm:cha}
Let $s>0$. Let $u$ be any solution of the equation
$$
 \big( -\mathcal{L}+\partial_\rho^2 +\frac{1-2s}{\rho}\partial_\rho +\frac{1}{4}\rho^2 \Delta_z \big)u(v,z,\rho) = 0
$$
and let $1\le p\le \infty$. Then $u$ satisfies the uniform estimates
\begin{equation*}
\int_N |u(v,z,\rho)|^p \,dv\, dz \leq  C, \qquad \rho > 0
\end{equation*}
if and only if $u=c\rho^{2s}f\ast \varphi_{s,\rho}$ for some $f\in L^p(N)$, $1<p\le \infty$. When $p=1$, $u=c\rho^{2s}\mu\ast \varphi_{s,\rho}$, for a complex Borel measure $\mu$. Moreover, when $1<p<\infty$, $u\to c' f$ in $L^p(N)$ as $\rho\to0$ (for $p=1$, $u\to c' \mu$ weakly as $\rho\to0$). 
\end{thm}

Later we will show that $u=C_1(n,m,s)\rho^{2s}f\ast \varphi_{s,\rho}$ converges to $f$ in $L^p$, $1\le p<\infty$ as $\rho\to0$, where
\begin{equation}
\label{eq:C1}
C_1(n,m,s)=\frac{4^m}{\pi^{n+m/2}}\frac{\Gamma(n+s)\Gamma\big(\frac{n+m+s}{2}\big)}{\Gamma(s)\Gamma(\frac{n+s}{2})}.
\end{equation}
 In the case of the Laplacian $\Delta$ on $\R^n$, the function 
$$
u(x,\rho)=C\rho^s\int_{\R^n}f(x-y)(\rho^2+|y|^2)^{-\frac{n+s}{2}}\,dy
$$
solves the extension problem associated to the Laplacian with initial condition $f$. Moreover, it has been proved (see for instance \cite{Caffarelli-Silvestre}) that 
$$
-\lim_{\rho\to0}\rho^{1-s}\partial_{\rho}u(x,\rho)=C_s(-\Delta)^{s/2}f(x).
$$
In a similar way we can obtain conformally invariant fractional powers $\mathcal{L}_s$ of the sublaplacian as the Dirichlet-to-Neumann map associated to the extension problem \eqref{eq:ep}. 

Before stating our result, let us recall the definition of $\mathcal{L}_s$ in the context of $H$-type groups (see \cite{RT} for the case of Heisenberg groups). The operator $\mathcal{L}$ is self adjoint and admits an explicit spectral resolution. In fact, $\mathcal{L}$ commutes with $\Delta_z$ and hence there is a joint spectral theory of $\mathcal{L}$ and $\Delta_z$. Using this, the operator $\mathcal{L}_s$ is simply defined by 
$$
\mathcal{L}_s=2^s(-\Delta_z)^{s/2}\frac{\Gamma\big(\frac{\mathcal{L}(-\Delta_z)^{-1/2}+1+s}{2}\big)}{\Gamma\big(\frac{\mathcal{L}(-\Delta_z)^{-1/2}+1-s}{2}\big)}.
$$
In view of Stirling's formula for the Gamma function we see that $\mathcal{L}_s$ is essentially the pure power  $\mathcal{L}^s$. However, as explained in \cite{RT}, it is more convenient to work with $\mathcal{L}_s$ rather than $\mathcal{L}_s$. Moreover, it has the ``conformal invariance'', see e.g. (1.9) in \cite{FGMT}.

We now have the following theorem which obtains $\mathcal{L}_sf$ as the Dirichlet-to-Neumann map associated to the extension problem.

\begin{thm}
\label{thm:sol}
Let $f\in L^p(N)$, $1\le p<\infty$. Then, as $\rho\to0^+$,
\begin{equation}
\label{eq:convergfT}
u=C_1(n,m,s)\rho^{2s}f\ast \varphi_{s,\rho} \rightarrow  f \qquad \text{ in } L^p(N),
\end{equation}
where $\varphi_{s,\rho}$ is defined in \eqref{eq:varfisr} and $C_1(n,m,s)$ is the one in \eqref{eq:C1}.
 If we further assume that $\mathcal{L}_sf\in L^p(N)$ then
\begin{equation} 
\label{eq:limits}
-\lim_{\rho\to 0^+}\rho^{1-2s} \partial_\rho (u(v,z,\rho)) = 2^{1-2s}\frac{\Gamma(1-s)}{\Gamma(s)} \mathcal{L}_s f(v,z).\end{equation}
Moreover, when $0<s<1/2$, we also have the pointwise representation
\begin{equation}
\label{eq:ir}
\mathcal{L}_sf(x)=\frac{4^{m+s} }{\pi^{n+m/2}}\frac{\Gamma(n+s)\Gamma\big(\frac{n+m+s}{2}\big)}{\Gamma\big(\frac{n+s}{2}\big)|\Gamma(-s)|}\int_{N}\frac{f(x)-f(y)}{|xy^{-1}|^{Q+2s}}\,dy
\end{equation}
for all $f\in C^1(N)$ such that $X_jf, Z_kf\in L^{\infty}(N)$, $j=1,\ldots, 2n$, $k=1,\ldots, m$.
Finally, for $0<s<1/2$, the following limit also exists in the $L^p(N)$ sense:
\begin{equation}
\label{eq:limit2}
-\lim_{\rho\to0^+}\frac{u(v,z,\rho)-f(v,z)}{\rho^{2s}} =\frac{|\Gamma(-s)|}{4^s\Gamma(s)}\mathcal{L}_s f(v,z).
\end{equation}
\end{thm}

In our earlier paper \cite{RT}, we have obtained the integral representation for $\mathcal{L}_sf$ in the context of the Heisenberg group. (We take this opportunity to correct an error in the previous work: there we have mentioned that the integral representation is valid for all $0<s<1$. But in fact, we had proved it only for $0<s<1/2$ as it is written. For higher values of $s$, we need to substract more terms from the Taylor expansion of $f$, see for instance \cite[p. 9]{Si}). In \cite{RT} we have used the integral representation in order to get the ground state representation for the fractional power $\mathcal{L}_s$ which has led to a Hardy-type inequality with non-homogeneous weight function. In this paper we obtain another proof of the Hardy inequality, in the general context of $H$-type groups which is based on the so called trace Hardy inequality. In order to state the inequality we introduce the following variant of a Sobolev space. Let $\nabla u=(X_1u, \ldots, X_{2n}u,  \frac12Z_1u, \ldots, \frac12Z_mu,\partial_\rho u)$. Let $\widetilde{W}_0^{s,2}(S)$ be the completion of $C_0^{\infty}(N\times \R)$ with respect to the norm 
$$
\|u\|_{(s)}^2=\int_0^{\infty}\int_N|\nabla u(v,z,\rho)|^2\rho^{1-2s}\,dv\,dz\,d\rho.
$$
This is indeed a norm: the vanishing of the integral implies the vanishing of the gradient and hence the function reduces to a constant. But then, as the function is from $ C_0^\infty $, it has to be zero.

\begin{thm}[General trace Hardy inequality]
\label{thm:GtHi}
Let $0<s<1$, $\rho>0$ and let $\varphi$ be a real valued function in the domain of $\mathcal{L}_s$. Further assume that $ \varphi^{-1} \mathcal{L}_s\varphi $ is locally integrable. Then for any real valued function $u \in \widetilde{W}_0^{s,2}(S)$, we have the inequality
$$
\int_0^{\infty}\int_N|\nabla u(v,z,\rho)|^2\rho^{1-2s}\,dv\,dz\,d\rho\ge \frac{2^{1-2s}\Gamma(1-s)}{\Gamma(s)}\int_Nu^2(v,z,0)\frac{\mathcal{L}_s\varphi(v,z)}{\varphi(v,z)}\,dv\,dz.
$$
 Moreover, the inequality is sharp and equality holds whenever $u$ is a solution of the extension problem with initial condition $\varphi$.
\end{thm}
There is no problem in proving the inequality in Theorem \ref{thm:GtHi} for $ C_0^\infty $ functions and hence for $ u $ in our space $\widetilde{W}_0^{s,2}(S)$. But for the sharpness, we have to show that the solution of the extension problem belongs to our space. A priori this is not clear; this is proved in Theorem ~\ref{thm:Ws2N}.

Let 
\begin{equation}
\label{eq:C2}
C_2(n,m,s)=4^{2s}\frac{\Gamma\big(\frac{n+1+s}{2}\big)}{\Gamma\big(\frac{n+1-s}{2}\big)}\frac{\Gamma\big(\frac{n+m+s}{2}\big)}{\Gamma\big(\frac{n+m-s}{2}\big)}.
\end{equation}
By taking $u$ to be a solution of the extension problem with initial condition $f$, $\varphi=\varphi_{-s,\delta}$, $\delta>0$, and making use of a result of Cowling and Haagerup \cite{CH}, we can obtain the following Hardy type inequality.

\begin{cor}
\label{cor:HarNH}
Let $0<s<1$, $\delta>0$. Let $f$, $\mathcal{L}_sf\in L^2(N)$. Then
$$
(\mathcal{L}_sf,f)\ge C_2(n,m,s)\delta^{2s}\int_Nf^2(v,z)\big((\delta^2+|v|^2)^2+16|z|^2\big)^{-s}\,dv\,dz,
$$
where $C_2(n,m,s)$ is the constant in \eqref{eq:C2}.
Here again the constant is sharp and equality is obtained when $f=\varphi_{-s,\delta}$.
\end{cor}

In the Euclidean case, sharp Hardy inequalities with non-homogeneous as well as homogeneous weight functions are known and they play an important role in several problems of partial differential equations. On the other hand, trace Hardy inequality for the Laplacian plays a role in proving Hardy inequality for fractional powers of $\Delta$ on $\R^n$. In the same way, we can prove the following version of Hardy inequality for the fractional powers of $\mathcal{L}$ on $N$. 
Let $
\varphi_s(v,z)=|(v,z)|^{-(n+m+s)}$ where $|(v,z)|=(|v|^4+16|z|^2)^{\frac14}$ is the homogeneous norm on $N$ and $\psi_s(v,z)=C_1(n,m,s)(\varphi_s\ast |\cdot|^{-Q+2s})(v,z)$, where $C_1(n,m,s)$ is given in \eqref{eq:C1} and let us define
\begin{equation}
\label{eq:homowsbis}
w_s(v,z)=\varphi_s(v,z)\psi_s(v,z)^{-1}.
\end{equation} 
It is easily verified that  $w_s$ is homogeneous of degree $-2s$.

\begin{thm}
\label{thm:tHihomo}
Let $0<s<1$ and let $u\in \widetilde{W}_0^{s,2}(S)$ be a real valued function. Then we have 
$$
\int_0^{\infty}\int_N|\nabla u(v,z,\rho)|^2\rho^{1-2s}\,dv\,dz\,d\rho\ge 2^{1-2s}\frac{\Gamma(1-s)}{\Gamma(s)}C_2(n,m,s)\int_Nu^2(v,z,0)w_s(v,z)\,dv\,dz,
$$
where $w_s$ is the homogeneous weight in \eqref{eq:homowsbis} and $C_2(n,m,s)$ is the constant in \eqref{eq:C2}.
The constant is sharp but equality is never achieved in $\widetilde{W}_0^{s,2}(S)$.
\end{thm}

\begin{cor}
\label{thm:Hardyhom}
Let $0<s<1$, and $f$, $\mathcal{L}_sf \in L^2(N)$. Then
$$
(\mathcal{L}_sf,f)\ge C_2(n,m,s)\int_N f^2(v,z)w_s(v,z)\,dv\,dz,
$$
where $w_s$ is the function \eqref{eq:homowsbis}, which is homogeneous of degree $-2s$ and $C_2(n,m,s)$ is given in \eqref{eq:C2}. The constant is sharp but equality is never achieved in $\widetilde{W}_0^{s,2}(S)$.
\end{cor}

\begin{rem}
It would be interesting to see if $w_s(v,z)$ can be replaced by $|(v,z)|^{2s}$. 
\end{rem}

As explained earlier, the extension problem for the sublaplacian on Heisenberg groups $ \H^n$ has been studied by Frank et al in \cite{FGMT}, where the authors have proved the existence of a solution when the initial condition $ f $ belongs to $ C_0^\infty(\H^n).$ However, the solution is not given explicitly in terms of the initial condition. In this article, as shown in Theorem \ref{thm:sol}, we are able to write down the solution explicitly. This allows us to study $ L^p $ convergence of the solution and its $ \rho $ derivative. The analogue of Theorem \ref{thm:sol} in the context of the Euclidean Laplacian was proved by Caffarelli and Silvestre in \cite{Caffarelli-Silvestre}.

As mentioned above, the explicit solution has also been obtained by M\"ollers et al \cite{MOZ} by considering the operator $L_s$ in \eqref{eq:Ls}. They have shown that solutions of the equation $L_s u=0$ with initial condition $u(z,0,t)=f(z,t)$ which are radial in the $\zeta$ variable are given by $u(z,\zeta,t)=c_{n,s}|\zeta|^{2s}f \ast \varphi_{s,|\zeta|}(z,t)$ where $\varphi_{s,\rho}$ is the one in \eqref{eq:varfisr} when $N=\H^n$. Using representation theory arguments, they have proved a nice isometry property for the operator $P_s$ which takes $f$ into the above solution $u(z,\zeta,t)$. In this paper we prove an analogous property for the solution operator associated to $H$-type groups. In the process, we also provide an alternate proof of the result of M\"ollers et al, see Subsection \ref{subsec:iso}. First we need to set up some notation. For $s>0$, let $H^s(N)$ stand for the domain of $\mathcal{L}_s^{1/2}$: that is to say, $f\in H^s(N)$ if and only if $\mathcal{L}_s^{1/2}f\in L^2(N)$ which is the same as saying that $ \mathcal{L}_{s/2}f \in L^2(N)$. Note that this space is just a variant of the homogeneous Sobolev space associated to the sublaplacian. In the Euclidean case, the correspoding spaces $ H^s(\R^n) $ are defined in terms of the fractional powers $ (-\Delta)^{s/2}$. We write $\|f\|_{H^s(N)} = \|\mathcal{L}_s^{1/2}f\|_{L^2(N)}$ to denote the norm on $H^s(N)$. For each $ \omega \in \mathbb{S}^{m-1}$ we denote by $R_{\omega}f $ the Radon transform of an integrable function $f$ on $ N $ in the $z$ variable. That is
\begin{equation}
\label{eq:Radon}
R_{\omega}f(v,t)=\int_{z'\in\omega^{\perp}}f(v,t\omega+z')\,dz', \quad t\in \R,
\end{equation}
where $\omega^{\perp}=\{z': z'\cdot \omega=0\}$ and $dz'$ is the $(m-1)$-dimensional Lebesgue measure on $\omega^{\perp}$. As will be explained later (see Subsection \ref{subsec:reprH}), for each $\omega$, $R_{\omega}f(v,t)$ can be considered as a function on a $H$-type group $\H_{\omega}^n$ of dimension $(2n+1)$ which is isomorphic to $\H^{n}$. The space $H^s(\H_{\omega}^n)$ is defined using the sublaplacian $\mathcal{L}$ on $\H^n$. The following is the isometry property for the solution operator associated to the extension problem on the group $N$. 

\begin{thm}
\label{thm:MOZgen}
Let $0<s<n+1$. Then any solution $u$ of the extension problem with initial condition $f$ satisfies
$$
\int_{\mathbb{S}^{m-1}}\|R_{\omega}(-\Delta_{z})^{\frac{m-1}{4}}u\|^2_{H^{s+1}(\H^{n+1})}\,d\sigma(\omega)=C \|f\|_{H^s(N)}^2.
$$
\end{thm}

The plan of the paper is the following. In Section \ref{sec:prelim} we recall facts related to $H$-type Lie algebras and groups and we describe the representation theory of such groups. We also define the sublaplacian on a $H$-type group and its fractional powers, and recall a known result related to solvable extensions of $H$-type groups. The study of the extension problem for the sublaplacian on $H$-type groups will be addressed in Section \ref{sec:ep}. More precisely, we will prove Theorems \ref{thm:cha}, \ref{thm:sol} and \ref{thm:MOZgen}. We will also perform the higher order extension problem for values of $s>0$. Finally, in Section \ref{sec:Hardy} we will prove trace Hardy and Hardy inequalities stated in Theorems \ref{thm:GtHi} and \ref{thm:tHihomo} and Corollaries \ref{cor:HarNH} and \ref{thm:Hardyhom}.

\section{Preliminaries on $H$-type groups}
\label{sec:prelim}

\subsection{$H$-type Lie algebras and groups}
\label{subsec:Htype}

A step two nilpotent Lie group $N$ is said to be a $H$-type group if its Lie algebra $\mathfrak{n}$ is of $H$-type. A Lie algebra $\mathfrak{n}$ is said to be a $H$-type Lie algebra if we can write $\mathfrak{n}$ as the direct sum $\mathfrak{v}\oplus \mathfrak{z}$ of two Euclidean spaces with a Lie algebra structure such that $\mathfrak{z}$ is at the centre of $\mathfrak{n}$ and for every unit vector $v\in\mathfrak{v}$ the map $\operatorname{ad}(v)$ is a surjective isometry of the orthogonal complement of $\operatorname{ker}(\operatorname{ad}(v))$ onto $\mathfrak{z}$. If $\mathfrak{n}$ is such a $H$-type algebra we define a map $J:z\to \operatorname{End}(\mathfrak{v})$ by
$$
(J_\omega v,v')=(\omega,[v,v']), \qquad \omega\in \mathfrak{z}, \quad v,v'\in \mathfrak{v}.
$$ 
It then follows that $J_\omega^2=-I$ whenever $\omega$ is a unit vector in $\mathfrak{z}$. We can therefore introduce a complex structure on $\mathfrak{v}$ using $J_\omega$. The Hermitian inner product on $\mathfrak{v}$ is given by 
$$
\langle v,v'\rangle_\omega=(v,v')+i(J_\omega v,v')=(v,\omega)+i([v,v'],\omega).
$$
Thus when $N$ is a $H$-type group, identifying $N$ with its Lie algebra $\mathfrak{n}$, we write the elements of $N$ as $(v,z)$, $v\in \mathfrak{v}$, $z\in \mathfrak{z}$. In view of the Baker--Campbell--Hausdorff formula, the group law takes the form
$$
(v,z)(v',z')=(v,z)+(v',z')+\frac12[(v,z),(v',z')].
$$
The best known example of a $H$-type group is the Heisenberg group $\H^n=\R^{2n}\times \R$. By identifying $\R^{2n}$ with $\C^n$, we write the elements of $\H^n$ as $(v,t)$, $v\in \C^n$, $t\in \R$. The group law in $\H^n$ is then given by 
$$
(v,t)(v',t')=\big(v+v',t+t'+\frac12\Im v\cdot \bar{v}\big).
$$
The Heisenberg groups play an important role in studying problems on $H$-type groups. This is due to the fact that to every $H$-type Lie algebra $\mathfrak{n}=\mathfrak{v}\oplus \mathfrak{z}$ and unit vector $\omega\in \mathfrak{z}$ we can associate a Heisenberg Lie algebra $\mathfrak{h}_\omega$ as follows. Given a unit vector $\omega\in \mathfrak{z}$, let $k(\omega)$ stand for the orthogonal complement of $\omega$ in $\mathfrak{z}$. Then the quotient algebra $\mathfrak{n}(\omega)=\mathfrak{n}/k(\omega)$ can be identified with $\mathfrak{v}\oplus \R$ by defining
$$
[(v,t),(v',t')]_\omega=(0,[J_\omega v,v']).
$$
It is known (see \cite{KR,MuSe}) that this algebra is isomorphic to the Heisenberg algebra $\mathfrak{h}^n$. We denote the corresponding group by $\H_\omega^n$ which we often identify with $\H^n$.

\subsection{The representation theory of $H$-type groups}
\label{subsec:reprH}

Before describing the representation theory of $H$-type groups, let us first recall some facts about irreducible unitary representations of the Heisenberg groups $\H^n$. It is well known that any irreducible unitary representation of $\H^n$ which is nontrivial at the centre (namely on $\{0\}\times \R$) is unitarily equivalent to the Schr\"odinger representation $\pi_{\lambda}$, for a unique $\lambda\in \R^{*}=\R\setminus\{0\}$. Here these representations $\pi_{\lambda}$ are all realised on $L^2(\R^n)$ and given explicitly by 
$$
\pi_{\lambda}(v,t)\varphi(\xi)=e^{i\lambda t}e^{i(x\cdot \xi+\frac12 x\cdot y)}\varphi(\xi+y)
$$
where $v=x+iy$, $\varphi\in L^2(\R^n)$. There is another family of one dimensional representations which do not play any role in the Plancherel theorem. Hence we do not attempt to describe them.

The group Fourier transform of an $L^1(\H^n)$ function $f$ is defined to be the operator valued function $\lambda\to \widehat{f}(\lambda)$ given by 
$$
\widehat{f}(\lambda)=\int_{\H^n} f(v,t)\pi_{\lambda}(v,t)\,dv\,dt.
$$
Sometimes we use the notation $\pi_{\lambda}(f)$ instead of $\widehat{f}(\lambda)$. Recalling the definition of $\pi_{\lambda}$ it is easy to see that 
$$
\widehat{f}(\lambda)=\int_{\C^n} f^{\lambda}(v)\pi_{\lambda}(v,0)\,dv
$$
where we have written $f^{\lambda}$ to stand for 
$$
f^{\lambda}(v)=\int_{-\infty}^{\infty}e^{i\lambda t}f(v,t)\,dt,
$$
the Euclidean inverse Fourier transform of $f$ in the central variable. We will be using this notation without any further comments.

When $f\in L^1\cap L^2(\H^n)$ it can be easily verified that $\widehat{f}(\lambda)$ is a Hilbert--Schmidt operator and we have 
$$
\int_{\H^n}|f(v,t)|^2\,dv\,dt=(2\pi)^{-n-1}\int_{-\infty}^{\infty}\|\widehat{f}(\lambda)\|_{\operatorname{HS}}^2|\lambda|^n\,d\lambda.
$$
The above equality of norms allows us to extend the definition of the Fourier transform to all $L^2$ functions. It then follows that we have Plancherel theorem: $f\to \widehat{f}$ is a unitary operator from $L^2(\H^n)$ onto $L^2(\R^*, \textrm{S}_2, d\mu)$ where $ \textrm{S}_2$ stands for the space of all Hilbert--Schmidt operators on $L^2(\R^n)$ and $d\mu(x)=(2\pi)^{-n-1}|\lambda|^nd\lambda$ is the Plancherel measure for the group $\H^n$.

The connection between $H$-type Lie algebras and Heisenberg Lie algebras allows us to get a quick picture of the representation theory of $H$-type groups. As in the case of the Heisenberg groups, the irreducible unitary representations of $H$-type group $N$ comes in two groups. As before we neglect the one dimensional representations which are trivial on the centre of $N$. If $\pi$ is any infinite dimensional irreducible representation of $N$, then its restriction to the centre has to be a unitary character. This means that $\exists \lambda\in \R^*$ and $\omega\in \mathbb{S}^{m-1}$, the unit sphere in the centre (identified with $\R^m$) such that $\pi(0,z)=e^{i\lambda(\omega,z)}\operatorname{Id}$. It can be shown that such a representation factors through a representation of $\H_{\omega}^n$, the group introduced in Subsection \ref{subsec:Htype}. By making use of the Stone--von Neumann theorem we can show that all irreducible unitary representations of $N$ are parametrised by $(\lambda,\omega)$, $\lambda>0$, $\omega\in \mathbb{S}^{m-1}$. We denote such a representation by $\pi_{\lambda,\omega}$. It follows that the restriction of $\pi_{\lambda,\omega}$ to $\H_{\omega}^n$ is unitarily equivalent to the Schr\"odinger representation $\pi_{\lambda}$. The Plancherel theorem for $H$-type groups $N$ reads as
$$
\|f\|_2^2=(2\pi)^{-n-m}\int_0^{\infty}\Big(\int_{\mathbb{S}^{m-1}}\|\pi_{\lambda,\omega}(f)\|_{\operatorname{HS}}^2\,d\sigma(\omega)\Big)\lambda^{n+m-1}\,d\lambda.
$$
This theorem can be deduced from the Plancherel theorem for the Heisenberg group by means of partial Radon transform.

We now briefly recall this connection which will be made use of later. Let $R_{\omega}f$ be the Radon transform defined in \eqref{eq:Radon}. We identify $R_{\omega}f$ with a function on $\H^n_{\omega}$. It can be shown that 
$$
R_{\omega}(f\ast g)=R_{\omega}f\ast_{\omega}R_{\omega}g
$$
for two functions $f,g\in L^1(N)$. In the above, the convolution on the left is on the group $N$ whereas $\ast_{\omega}$ on the right stands for the convolution on the Heisenberg group $\H^n_{\omega}$.
Using the above relation and the connection between $\pi_{\lambda,\omega}$ and $\pi_{\lambda}$ we can show that 
$$
\pi_{\lambda,\omega}(f)=\pi_{\lambda}(R_{\omega}f), \qquad \omega\in \mathbb{S}^{m-1}, \quad \lambda>0.
$$
From this relation and Plancherel theorem for $\H^n_{\omega}$ we can deduce Plancherel for the group $N$. 

We say that a function $f$ on $N$ is radial if it is radial in the $v$ variable. For such functions, the Fourier transform is given by a simple formula. Let $H(\lambda)=-\Delta+\lambda^2|x|^2$ be the scaled Hermite operator  with spectral decomposition 
$$
H(\lambda)=\sum_{k=0}^{\infty}(2k+n)|\lambda|P_k(\lambda).
$$
Here $P_k(\lambda)$ are the orthogonal projections of $L^2(\R^n)$ onto the $k$-th eigenspaces corresponding to the eigenvalues $(2k+n)|\lambda|$. For $\alpha\in \Na^n$ and $\lambda \neq 0$, let $\Phi_{\alpha}^{\lambda}$ be the Hermite functions on $\R^{n}$ which are eigenfunctions of $H(\lambda)$ with eigenvalues $(2|\alpha|+n)|\lambda|$ where $|\alpha|=\sum_{j=1}^n\alpha_j$, see for instance \cite{STH}. If $f$ is a radial function on $N$, then its Fourier transform $\pi_{\lambda, \omega}(f)$ reduces to a function of $H(\lambda)$:
$$
\pi_{\lambda, \omega}(f)=\sum_{k=0}^{\infty}\widehat{f}(\lambda\omega, k)P_k(\lambda)
$$
where the coefficients $\widehat{f}(\lambda\omega, k)$ are given by the following formula. Let 
$$
f^{\lambda \omega}(v)=\int_{\R^m}e^{i\lambda\omega \cdot z}f(v,z)\,dz
$$
be the inverse Fourier transform of $f$ in the central variable at $\lambda \omega$, $\lambda>0$, $\omega\in \mathbb{S}^{m-1}$. Let 
$$
\varphi_{k}^{\lambda}(v)=L_{k}^{n-1}\big(\frac12\lambda|v|^2\big)e^{-\frac{\lambda}4|v|^2}
$$
stand for the Laguerre functions of type $(n-1)$ on $\C^n$. Then
$$
\widehat{f}(\lambda\omega, k)=c_n\frac{k! (n-1)!}{(k+n-1)!}\int_{\C^n}f^{\lambda \omega}(v)\varphi_k^{\lambda}(v)\,dv.
$$
When $f$ is also radial in the $z$ variable, $f^{\lambda\omega}$ is independent of $\omega$ and we have 
$$
\pi_{\lambda, \omega}(f)=\pi_{\lambda, 1}(f)=\sum_{k=0}^{\infty}\widehat{f}(\lambda, k)P_k(\lambda).
$$
We will make use of this formula in calculating later the Fourier transform of $\varphi_{s,\rho}$ given in \eqref{eq:varfisr}.

\subsection{Sublaplacian $\mathcal{L}$ and its fractional powers}
\label{subsec:subl}
We now define the sublaplacian $\mathcal{L}$ on a $H$-type group $N$ which is the main focus of our work. We fix an orthonormal basis $X_j$, $j=1,2,\ldots,2n$ for the Lie algebra $\mathfrak{v}$ and $Z_j$, $j=1,2,\ldots, m$ for the Lie algebra $\mathfrak{z}$. We denote by $\Delta_z=\sum_{j=1}^mZ_j^2$ the ordinary Laplacian on the centre of $N$. The sublaplacian $\mathcal{L}$ on $N$ is defined by $\mathcal{L}=-\sum_{j=1}^{2n}X_j^2$. Note that we have defined $\mathcal{L}$ with a negative sign which makes it nonnegative. It is a subelliptic operator which has a self adjoint extension with an explicit spectral resolution. 
In the case of the Heisenberg group $\H^n$ the vector fields are given by 
$$
X_j=\frac{\partial}{\partial x_j}+\frac12y_j\frac{\partial}{\partial t}, \quad X_{n+j}=\frac{\partial}{\partial y_j}-\frac12x_j\frac{\partial}{\partial t}, \quad Z=\frac{\partial}{\partial t}
$$
for $j=1,2,\ldots,n$ where $(v,t)=(x+iy,t)\in \H^n$. More explicitly, the sublaplacian is given by 
$$
\mathcal{L}=-\Delta_{\C^n}-\frac14|v|^2\frac{\partial^2}{\partial t^2}+\sum_{j=1}^n\big(x_j\frac{\partial}{\partial y_j}-y_j\frac{\partial}{\partial x_j}\big)\frac{\partial}{\partial t}.
$$
For more about $\mathcal{L}$, we refer to \cite{RSt} and \cite{STH}. 

In studying the extension problem for the sublaplacian $\mathcal{L}$ we will be making good use of the heat kernel associated to $\mathcal{L}$. It is known that $\mathcal{L}$ generates a diffusion semigroup with a kernel $q_t$ from the Schwartz class: $e^{-t\mathcal{L}}f=f\ast q_t$ for $f\in L^p(N)$. An explicit expression for the heat kernel is given by 
$$
\int_{\R^m}e^{i\lambda z\cdot \omega}q_t(v,z)\,dz= (4\pi)^{-n} \Big(\frac{\lambda}{\sinh(t \lambda)}\Big)^{n} e^{-\frac{\lambda}{4} (\coth(t\lambda) |v|^2}
$$
for $\lambda>0$ and $\omega \in \mathbb{S}^{m-1}$, see \cite{Cy,Ra}. Good estimates on the heat kernel $q_t$ are known, see \cite{YZ}.

It is well known (see e.g. Strichartz \cite{RSt}) that an essential role is played by the Hermite operators $H(\lambda)=-\Delta+\lambda^2|x|^2$ in the joint spectral theory of the sublaplacian $\mathcal{L}$ and $\frac{\partial}{\partial t}$ on the Heisenberg group. In view of the relation $(\mathcal{L}f)^{\widehat{}}(\lambda)=\widehat{f}(\lambda)H(\lambda)$, the spectral theorem for $\mathcal{L}$ takes the form 
$$
\mathcal{L}f(v,t)=(2\pi)^{-n-1}\int_{-\infty}^{\infty} \operatorname{tr}\big(\pi_{\lambda}(v,t)^*\widehat{f}(\lambda)H(\lambda)\big)|\lambda|^n\,d\lambda.
$$

In calculating the trace we make use of the orthonormal basis $\Phi_{\alpha}^\lambda$, $\alpha\in \Na^n$  of $L^2(\R^n)$, introduced in Subsection \ref{subsec:reprH}, consisting of the Hermite functions which are eigenfunctions of the Hermite operator $H(\lambda)$ with eigenvalues $(2|\alpha|+n)|\lambda|$. Thus, 
$$
\mathcal{L}f(v,t)=(2\pi)^{-n-1}\int_{-\infty}^{\infty} \Big(\sum_{\alpha\in \Na^n}\big((2|\alpha|+n)|\lambda|\big)\big(\pi_{\lambda}(v,t)^*\widehat{f}(\lambda)\Phi_{\alpha}^{\lambda}, \Phi_{\alpha}^{\lambda}\big)\Big)|\lambda|^n\,d\lambda.
$$
From the above it is clear that the (pure) fractional powers of $\mathcal{L}$ can be defined by 
$$
\mathcal{L}^sf(v,t)=(2\pi)^{-n-1}\int_{-\infty}^{\infty} \Big(\sum_{\alpha\in \Na^n}\big((2|\alpha|+n)|\lambda|\big)^s\big(\pi_{\lambda}(v,t)^*\widehat{f}(\lambda)\Phi_{\alpha}^{\lambda}, \Phi_{\alpha}^{\lambda}\big)\Big)|\lambda|^n\,d\lambda.
$$
However, for several reasons it is convenient to work with the conformally invariant fractional powers  
$$
\mathcal{L}_sf(v,t)=(2\pi)^{-n-1}\int_{-\infty}^{\infty} \Big(\sum_{\alpha\in \Na^n}\frac{\Gamma\big(\frac{2|\alpha|+n+1+s}{2}\big)}{\Gamma\big(\frac{2|\alpha|+n+1-s}{2}\big)}(2|\lambda|)^s\big(\pi_{\lambda}(v,t)^*\widehat{f}(\lambda)\Phi_{\alpha}^{\lambda}, \Phi_{\alpha}^{\lambda}\big)\Big)|\lambda|^n\,d\lambda.
$$
Symbolically, we can write them as 
$$
\mathcal{L}_s=\frac{\Gamma\big(\frac{\mathcal{L}(-Z^2)^{-\frac12}+1+s}{2}\big)}{\Gamma\big(\frac{\mathcal{L}(-Z^2)^{-\frac12}+1-s}{2}\big)}2^s(-Z^2)^{s/2}, \qquad Z^2=\frac{\partial^2}{\partial t^2}.
$$
Observe that $\mathcal{L}_s$ is defined for all values of $s$ for which $\frac{(2k+n)+1\pm s}{2}$ is not a pole of $\Gamma(x)$ for any $k=0,1,2,\ldots$.

The conformally invariant fractional powers $\mathcal{L}_s$ have been studied by several authors (\cite{BFM, BOO, FL, JW}). In an earlier paper \cite{RT} we investigated Hardy type inequalities for $\mathcal{L}_s$ on Heisenberg groups. 
Apart from being conformally invariant, the fractional powers $\mathcal{L}_s$ have the added advantage of having explicit fundamental solutions. As proved in Cowling-Haagerup \cite{CH}, in the more general setting of $H$-type groups, a fundamental solution for $\mathcal{L}_s$ on the Heisenberg groups $\H^n$ is given by the function $C_{n,s}|(v,t)|^{-Q+2s}$ where $Q=2n+2$ is the homogeneous dimension of $\H^n$ and $|(v,t)|$ is the norm defined by $|(v,t)|^4=|v|^4+16t^2$. Note that $|(v,t)|$ is homogeneous of degree one under the nonisotropic dilation $(v,t)\to (\delta v, \delta^2 t)$, $\delta>0$. 

The fractional powers $\mathcal{L}_s$ of the sublaplacian on $H$-type groups $N$ are defined as in the case of $\H^n$ by
$$
\mathcal{L}_s=\frac{\Gamma\big(\frac{\mathcal{L}(-\Delta_z)^{-\frac12}+1+s}{2}\big)}{\Gamma\big(\frac{\mathcal{L}(-\Delta_z)^{-\frac12}+1-s}{2}\big)}2^s(-\Delta_z)^{s/2}, \qquad \Delta_z=\sum_{j=1}^mZ_j^2.
$$
The spectral resolution of $\mathcal{L}$ can be written down using the connection between $N$ and the Heisenberg groups $\H_{\omega}^n$, $\omega\in \mathbb{S}^{m-1}$. More explicitly, we have 
\begin{multline*}
\mathcal{L}_sf(v,z)\\=(2\pi)^{-n-m}\int_{0}^{\infty} \int_{\mathbb{S}^{m-1}}\Big(\sum_{\alpha\in \Na^n}\frac{\Gamma\big(\frac{2|\alpha|+n+1+s}{2}\big)}{\Gamma\big(\frac{2|\alpha|+n+1-s}{2}\big)}(2\lambda)^s\big(\pi_{\lambda,\omega}(v,z)^*\widehat{f}(\lambda)\Phi_{\alpha}^{\lambda}, \Phi_{\alpha}^{\lambda}\big)\Big)\,d\sigma(\omega)\lambda^{n+m-1}\,d\lambda.
\end{multline*}
The non-isotropic dilations on $N$ are also given by $(v,z)\to (\delta v,\delta^2 z)$, $\delta>0$ and the homogeneous dimension is given by $Q=2(n+m)$. 

\subsection{Solvable extensions of $H$-type groups}
\label{subsec:epsub}

As we have mentioned in the Introduction the connection between the extension problem for $\mathcal{L}$ on $N$ and the eigenfunction equation for $\Delta_S$ on $S=AN$ will be exploited. Recall that $A=\R^+$ acts on $N$ by non isotropic  dilations and hence we can form the semidirect product $S$. The left invariant vector fields on $S$ are given by $E_j=\sqrt{\rho}X_j$, $T_k=\rho Z_k$ and $H=\rho\partial_{\rho}$ for $j=1,2,\ldots, 2n$, $k=1,2,\ldots,m$. The left Haar measure on $S$ is given by $\rho^{-n-m-1}\,dv\,dz\,d\rho$. 

When $G$ is a connected, simply connected rank one semisimple Lie group with Iwasawa decomposition $G=KAN$, the nilpotent part $N$ turns out to be of $H$-type and Riemannian symmetric space $X=G/K$ can be identified with the solvable extension $AN$ of $N$. The Laplace--Beltrami operator $\Delta_X$ on $X$ is then just $\Delta_S$ on $AN$.

Eigenfunctions of the Laplace--Beltrami operator $\Delta_X$ on non compact Riemannian symmetric spaces, not necessarily of rank one, have been studied by Helgason \cite{H} and others culminating in the celebrated theorem of Kashiwara et al \cite{KKMOT}. Using this result, Ben Sa\"id et al \cite{BsOS} gave a simpler characterisation of eigenfunctions of $\Delta_X$ satisfying Hardy type conditions. In the general context of $AN$ groups when $N$ is not necessarily associated to a semisimple Lie group, there is no analogue of the theorem of Kashiwara et al. However, recently Damek and Kumar \cite{DK} have proved an analogue of the theorem of Ben Sa\"id et al for all solvable extensions of $H$-type groups.

By slightly modifying the standard notation used in the literature, let us write
$$
\mathcal{P}_s(v,z,\rho)=\rho^{\frac{n+m+s}{2}}\big((\rho^2+|v|^2)^2+16|z|^2\big)^{-\frac{n+m+s}{2}}
$$
for the Poisson kernel associated to the group $S=AN$.
\begin{thm}[Damek--Kumar]
\label{thm:DK}
Assume that $u$ is an eigenfunction of the operator $\Delta_S$ with eigenvalue $-\frac14 \big((n+m)^2-s^2\big)$, $s>0$. Then there exists $f\in L^p(N)$, $1<p\le \infty$ such that $u(v,z,\rho)=f\ast \mathcal{P}_s(v,z,\rho)$ if and only if $u$ satisfies the uniform estimates
$$
\big(\int_N|u(v,z,\rho)|^p\,dv\,dz\big)^{1/p}\le C\rho^{\frac{n+m-s}{2}}
$$
for all $\rho>0$.
\end{thm}

The proof of this result given in \cite{DK} makes use of the maximum principle for an operator which resembles the extension operator we study in this article. We remark that an independent characterisation of solutions of the extension problem will naturally lead to a characterisation of eigenfunctions of $\Delta_S$ on solvable extensions of $H$-type groups. 

\section{An extension problem for the sublaplacian}
\label{sec:ep}

\subsection{Characterisation of solutions of the extension problem: proof of Theorem~\ref{thm:cha}}
\label{sub:cha}

The solution of the extension problem \eqref{eq:ep} with initial condition $ f \in L^p(N) $ given by $ u(v,z,\rho) =C_1(n,m,s) \rho^{2s}f \ast \varphi_{s,\rho}(v,z) $ satisfies the uniform estimate $ \|u(\cdot,\rho)\|_p \leq \| f\|_p$ trivially. Indeed, we can use Young's inequality and the fact that 
$$
\rho^{2s}\varphi_{s,\rho}(v,z)=\rho^{-2(n+m)}\varphi_{s,1}(\rho^{-1}v,\rho^{-2}z)
$$ 
is an integrable function, and therefore an approximate identity.

We will now show that this property characterises all solutions of the extension problem with $ L^p $ initial condition. We will prove this by connecting solutions of the  extension problem with eigenfunctions of the Laplace--Beltrami operator $ \Delta_S $  given in \eqref{eq:LBel} on the solvable extension $ S $ of $ N.$

Let us recall some facts already explained in the Introduction. If $u $ is a solution of the extension problem \eqref{eq:ep} then it is easy to see that the function $ w(v,z,\rho) = u(2^{-1/2}v, 2^{-1}z, (2\rho)^{1/2}) $ solves the equation
$$ \big( -\mathcal{L}+\partial_\rho^2 + (1-s) \partial_\rho +\rho \Delta_z \big)w(v,z,\rho) = 0.$$
Recalling the definition of $ \Delta_S, $ another calculation shows that $ \widetilde{w}(v,z,\rho) = \rho^{\frac{(n+m-s)}{2}}w(v,z,\rho) $ solves the eigenvalue problem
$$  \Delta_S \widetilde{w}(v,z,\rho)  = -\frac{1}{4}((n+m)^2 -s^2) \widetilde{w}(v,z,\rho), \quad (v,z,\rho) \in S.$$
We also note that the estimate $ \|u(\cdot,\rho)\|_p \leq C \| f\|_p $ translates into the condition
$$
\rho^{-\frac{(n+m-s)}{2}p}\int_{N} |\widetilde{w}(v,z,\rho)|^p dv dz \leq C ,\quad  \rho > 0.
$$
As explained in Subsection \ref{subsec:epsub}, eigenfunctions of the Laplace-Beltrami operator $ \Delta_S $ satisfying the above conditions have been characterised by Ben Sa\"id et al for Iwasawa $N$ groups in \cite{BsOS} and by Damek-Kumar in \cite{DK} for general $H$-type groups. Using their results we can get the following result, thereby completing the proof of Theorem \ref{thm:cha}.

\begin{thm} 
\label{thm:char}
Let $ 1 < p \leq \infty.$ Then any solution of the extension problem \eqref{eq:ep} which satisfies the uniform estimates $ \|u(\cdot,\rho)\|_p \leq C, \rho > 0$ is of the form $ f \ast \varphi_{s,\rho} $ for some $ f \in L^p(N).$
\end{thm}

\subsection{Proof of Theorem \ref{thm:sol}
}
\label{sub:proofThsol}

In this subsection we provide a proof of Theorem \ref{thm:sol}. It will be a consequence of several facts that  will be discussed as we go  along in this subsection.

The first issue to address is the connection with the solution of the extension problem~\eqref{eq:ep}. A solution of the extension problem \eqref{eq:ep} can be written down explicitly by modifying the formula by Stinga and Torrea (see \cite{ST}) for solutions of the extension problem
\begin{equation}
\label{eq:epEc}
( \Delta +\partial_\rho^2 +\frac{1-s}{\rho}\partial_\rho)u(x,\rho) = 0, \qquad u(x,0) = f(x).
\end{equation} 
Let us explain this more precisely.
In \cite{ST} it has been shown that the function $ u $ defined by the equation
$$ u(x,\rho)  = \frac{\rho^s}{2^s \Gamma(s/2)}  \int_0^\infty  e^{-\frac{1}{4t}\rho^2} e^{t\Delta}f(x) t^{-s/2-1} dt $$ solves \eqref{eq:epEc}. Here, $ e^{t\Delta} $ stands for the heat semigroup generated by the Laplacian $ \Delta. $ In a similar way we can show that the function
\begin{equation}
\label{eq:solST}
 u(x,\rho) =\frac{\rho^s}{2^s \Gamma(s/2)}  \int_0^\infty  e^{-\frac{1}{4t}\rho^2} e^{-t\mathcal{L}}f(x) t^{-s/2-1} dt 
\end{equation}
defined in terms of the heat semigroup $ e^{-t\mathcal{L}}$ generated by the sublaplacian solves the extension problem 
\begin{equation}
\label{eq:epR} 
( -\mathcal{L} +\partial_\rho^2 +\frac{1-s}{\rho}\partial_\rho)u(x,\rho) = 0, \qquad  u(x,0) = f(x) .\end{equation}
However, we are interested in the extension problem \eqref{eq:ep},
which is different from the problem~\eqref{eq:epR}.

By modifying the Stinga-Torrea formula \eqref{eq:solST} we can also write down a solution of the extension problem \eqref{eq:ep}. Let  $ p_{t,s}(\rho,z) $ be the heat kernel associated to the generalised sublaplacian
$$ \mathcal{L}(s) := \partial_\rho^2+\frac{1+2s}{\rho}\partial_\rho+\frac{1}{4}\rho^2 \Delta_z $$ on $ \R^+ \times \R^m$.
Let $ q_t(v,z) $ stand for the heat kernel associated to $ \mathcal{L}$. Then it is known that
\begin{equation}
\label{eq:qt}
   \int_{\R^m} q_t(v,z) e^{i \lambda \omega \cdot z} dz = (4\pi)^{-n} \Big(\frac{\lambda}{\sinh(t \lambda)}\Big)^{n} e^{-\frac{1}{4}\lambda \coth(t\lambda) |v|^2}
\end{equation}
for $ \lambda > 0 $ and $ \omega \in S^{m-1}$.
A similar formula is valid for the heat kernel $ p_{t,s}(\rho,w) $ as well:
\begin{equation}
\label{eq:pts}
\int_{\R^m} p_{t,s}(\rho,z) e^{i  \lambda \omega \cdot z} dz = (4\pi)^{-s-1}\Big(\frac{\lambda}{\sinh(t \lambda)}\Big)^{s+1} e^{-\frac{1}{4}\lambda \coth(t\lambda) \rho^2}.
\end{equation}
Note that the kernels $p_{t,s}$ and $q_t$ are normalized in such a way  that 
$$
\int_0^{\infty}\int_{\R^m}p_{t,s}(\rho,z)\,d\rho\, dz=1=\int_Nq_{t}(v,z)\,dv\, dz.
$$
The solution of the extension problem \eqref{eq:ep} can be written down explicitly in terms of the function $ e^{-t\mathcal{L}}f.$ Indeed, we have the following analogue of the Stinga--Torrea formula.

\begin{thm}
\label{thm:STf} For $ f \in L^p(N)$, $1 \leq p \leq \infty $, a solution of the extension problem \eqref{eq:ep} is given by
\begin{equation*}
u(v,z,\rho) = \frac{4\pi^{s+1}}{\Gamma(s)}\rho^{2s}  \int_0^\infty \int_{\R^m}  p_{t,s}(\rho, w) (e^{-t\mathcal{L}}f)(v,z-w) dw \, dt.
\end{equation*}
As $ \rho $ tends to zero, the solution $ u(\cdot,\rho) $ converges to $ f $ in $ L^p(N)$ for $ 1 \leq p < \infty.$
\end{thm}

\begin{proof} Applying $ \mathcal{L} $ to the function $ u $ and noting that $  U(v,z,t) = e^{-t\mathcal{L}}f(v,z) $ satisfies the heat equation $ -\mathcal{L}U(v,z,t) = \partial_t U(v,z,t) $ we see that 
$$ \mathcal{L}u(v,z,\rho) = -\frac{4\pi^{s+1}}{\Gamma(s)}\rho^{2s}  \int_0^\infty \int_{\R^m}  p_{t,s}(\rho, w) \partial_t U(v,z-w,t) dw dt.$$ 
Integrating by parts in the $t $ variable we can transfer the $ t $ derivative to $ p_{t,s}(\rho,w)$ and since it satisfies the heat equation associated to $ \mathcal{L}(s) $ we obtain
$$  \mathcal{L}u(v,z,\rho) = \frac{4\pi^{s+1}}{\Gamma(s)}\rho^{2s}  \big(\partial_\rho^2+\frac{1+2s}{\rho}\partial_\rho+\frac{1}{4}\rho^2 \Delta_z  \big) \int_0^\infty \int_{\R^m}  p_{t,s}(\rho, w) U(v,z-w,t) dw  dt.$$
A simple calculation shows that 
$$ \rho^{2s}  \big(\partial_\rho^2+\frac{1+2s}{\rho}\partial_\rho+\frac{1}{4}\rho^2 \Delta_z  \big)V(v,z,\rho) =   \big(\partial_\rho^2+\frac{1-2s}{\rho}\partial_\rho+\frac{1}{4}\rho^2 \Delta_z  \big)(\rho^{2s}V(v,z,\rho))$$ 
for any function $V(v,z,\rho). $ This proves that $ u $ satisfies the extension problem. In order to prove the convergence of the solution to the initial condition we make use of several properties of the heat kernels associated to $ \mathcal{L} $ and $ \mathcal{L}(s).$
By defining 
\begin{equation} 
\label{eq:homo}
\Phi_{s,\rho}(v,z) =\frac{4\pi^{s+1}}{\Gamma(s)}\rho^{2s}  \int_0^\infty \int_{\R^m}  p_{t,s}(\rho, w) q_t(v,z-w) dw  dt
\end{equation}
we see that the solution $ u $ is given by the convolution 
\begin{equation} 
\label{eq:convo}
u(v,z) = f \ast \Phi_{s,\rho}(v,z). 
\end{equation}
From \eqref{eq:qt} and \eqref{eq:pts} it is easy to see that the heat kernels satisfy the homogeneity conditions 
$$ p_{t,s}(r\rho,r^2 z) = r^{-2m-2(s+1)}p_{r^{-2}t,s}(\rho,z),\qquad q_t(rv,r^2z)= r^{-2m-2n}q_{r^{-2}t}(v,z)$$  for $ r >0.$ 
Using this in \eqref{eq:homo} we observe that  $ \Phi_{s,\rho}(v,z) = \rho^{-2(n+m)} \Phi_{s,1}(\rho^{-1}v,\rho^{-2}z).$ Moreover, as $ p_{t,s} $ and $ q_t $ are Schwartz class functions, $ \Phi_{s,1} \in L^1(N).$ 
Therefore, $ \Phi_{s,\rho} $ is an approximate identity. Observe that 
$\|\Phi_{s,1}\|_1=1 $. Indeed, 
$$
\int_N\Phi_{s,1}(v,z)\,dv\,dz =\frac{4\pi^{s+1}}{\Gamma(s)}  \int_0^\infty \Big(\int_{\R^m}  p_{t,s}(1, z)\,dz\Big)\Big(\int_N q_t(v,z) \,dv\,dz\Big)\,dt.
$$
Since $\int_Nq_t(v,z)\,dv\,dz=1$ and 
$$
\int_{\R^m}p_{t,s}(1,z)\,dz=(4\pi)^{-s-1}t^{-s-1}e^{-\frac1{4t}},
$$
we have
$$
\int_N\Phi_{s,1}(v,z)\,dv\,dz=\frac{4\pi^{s+1}}{\Gamma(s)}  \int_0^\infty t^{-s-1}e^{-\frac1{4t}}\,dt=1.
$$
Hence we immediately get the $ L^p(N) $ convergence of $ u = f \ast \Phi_{s,\rho} $ to $ f $ as $ \rho $ tends to $ 0 $. 
\end{proof}

It is possible to calculate the kernel $ \Phi_{s,1}(v,z)$ explicitly. Using \eqref{eq:qt} and \eqref{eq:pts} we see that
\begin{align}  
\label{eq:Phi}
\notag\Phi_{s,1}(v,z)& = (4\pi)^{-n-s-1} (2\pi)^{-m} \int_0^\infty  \Big(\int_{\R^m}  e^{-i z \cdot x} \Big(\frac{|x|}{\sinh(t |x|)}\Big)^{n+s+1} e^{-\frac{1}{4}|x| \coth(t|x|) (1+|v|^2)} \,dx \Big)\,  dt\\
&= (4\pi)^{-n-s-1} (2\pi)^{-m} \int_{\R^m}    \Big(\int_0^\infty\frac{|x|}{\sinh(t |x|)}\Big)^{n+s+1} e^{-\frac{1}{4}|x| \coth(t|x|) (1+|v|^2)}\,  dt  \Big) e^{-i z \cdot x}\,dx.
\end{align}
In the case of the Heisenberg group $ \H^n = \C^n \times \R $ we have computed the above integral explicitly in \cite[Proposition 4.2]{RT}. We recall the result in the following Theorem \ref{thm:com} for the convenience of the reader.

\begin{thm}[Roncal--Thangavelu]
\label{thm:com} 
For $ (v,w) \in \H^n $ and $ 0< s < 1 $ we have
$$  \int_0^\infty  \int_{-\infty}^\infty   e^{-i \lambda w} \Big(\frac{\lambda}{\sinh(t \lambda)}\Big)^{n+s+1} e^{-\frac{1}{4}\lambda \coth(t \lambda) (1+|v|^2)} d\lambda dt
= c_{n,s} ((1+|v|^2)^2+16w^2)^{-\frac{(n+1+s)}{2}},$$
where the constant $c_{n,s}$ is given by
$$
c_{n,s}=2^{n-1+3s}\pi^{-n-1}\Gamma\Big(\frac{n+s+1}{2}\Big)^2.
$$
\end{thm}

Since the Fourier transform of a radial function is radial and given by the Hankel transform we see that
$$  \Phi_{s,1}(v,z) = c_{m,s}  \int_0^\infty  \int_{0}^\infty   \frac{J_{m/2-1}(\lambda |z|)}{(\lambda |z|)^{m/2-1}} \Big(\frac{\lambda}{\sinh(t \lambda)}\Big)^{n+s+1} e^{-\frac{1}{4}\lambda \coth(t \lambda) (1+|v|^2)} \lambda^{m-1}d\lambda dt.$$
It may be possible to evaluate this integral directly but we use Radon transform to prove the following result giving us an explicit expression for $ \Phi_{s,1}.$

\begin{thm}
\label{thm:Phi1} For $ (v,z) \in N $ and $ 0< s < 1 $ we have
$$  \Phi_{s,1}(v,z) = c_{n,m,s} \big((1+|v|^2)^2+16|z|^2\big)^{-\frac{(n+m+s)}{2}}.$$
\end{thm}
\begin{proof}
It is known (and easy to check) that, for any $f\in L^1(\R^m)$,
$$
 \int_{-\infty}^\infty  e^{-i\lambda w} R_{\omega}f(w)\, dw = \int_{\R^m}  e^{-i \lambda  \omega\cdot x} f(x) \,dx.
$$
 In other words, 
\begin{equation}
\label{eq:R}
 R_{\omega}f(w)=(2\pi)^{-1}\int_{-\infty}^{\infty}\Big(\int_{\R^m}e^{-i  \lambda  \omega\cdot x}f(x)\,dx\Big)e^{i\lambda w}\,d\lambda.
 \end{equation}
By the uniqueness theorem for Radon transform, our theorem will follow once we show that the Radon transform of $ \Phi_{s,1} $ coincides with that of the function 
$((1+|v|^2)^2+16|z|^2)^{-\frac{(n+m+s)}{2}}$ up to a multiplicative constant. In view of \eqref{eq:Phi} and \eqref{eq:R} we know that 
$$
R_{\omega}\Phi_{s,1}(v,w)=C(2\pi)^{-1}\int_{-\infty}^{\infty}\int_{0}^{\infty} e^{-i \lambda w} \Big(\frac{\lambda}{\sinh(t \lambda)}\Big)^{n+s+1} e^{-\frac{1}{4}\lambda \coth(t \lambda) (1+|v|^2)}\, d\lambda \,dt
$$
which has been evaluated in Theorem \ref{thm:com}. Hence, we only need to verify that
$$  \int_{ x \cdot \omega = t} ((1+|v|^2)^2+16|x|^2)^{-\frac{(n+m+s)}{2}} d\mu(x) =c_{n,m,s}((1+|v|^2)^2+16 t^2)^{-\frac{(n+1+s)}{2}}$$
where $
c_{n,m,s}=C_1(n,1,s)C_1(n,m,s)^{-1}$,
which is rather easy to see. This completes the proof.
\end{proof}

In this article, we need the exact values of integrals of several kernels. These are collected in the following technical lemma.
\begin{lem}
\label{lem:I}
Let $n,m>0$ and let $j,\alpha$ be such that $\alpha-j>0$ and $\alpha>-n$. Then we have
\begin{equation}
\label{eq:I}
\int_N(1+|v|^2)^j\big((1+|v|^2)^2 +16 |z|^2\big)^{-\frac{n+m+\alpha}{2}}\,dv\,dz=\frac{\pi^{n+m/2}}{4^m}\frac{\Gamma(\alpha-j)\Gamma(\frac{n+\alpha}{2})}{\Gamma(n-j+\alpha)\Gamma\big(\frac{n+m+\alpha}{2}\big)}.
\end{equation}
\end{lem}
\begin{proof}
Let us call $I$ the integral in the left hand side of \eqref{eq:I}.
Recall the formula for the Beta function
\begin{equation*}
\int_0^{\infty}(1+t)^{-b}t^{a-1}\,dt=\frac{\Gamma(a)\Gamma(b-a)}{\Gamma(b)}, \qquad \Re (a)>0, \quad \Re (b-a)>0,
\end{equation*}
and let $\omega_n$ denote the area of the unit sphere $\mathbb{S}^{n-1}$, i.e., $\omega_n=\frac{2\pi^{n/2}}{\Gamma(n/2)}$.
By using polar coordinates in both variables, we have 
\begin{align*}
I&=\omega_{2n}\int_{\R^m}\int_0^{\infty}(1+r^2)^j\big((1+r^2)^2+16|z|^2\big)^{-\frac{(n+m+\alpha)}{2}}r^{2n-1}\,dr\,dz\\
&=\frac{\omega_{2n}}{2}\int_0^{\infty}\int_0^{\infty}(1+r)^j\big((1+r)^2+16|z|^2\big)^{-\frac{(n+m+\alpha)}{2}}r^{n-1}\,dr\,dz\\
&=\frac{\omega_{2n}}{2}\int_{\R^m}\int_0^{\infty}(1+r)^{-(n+m-j+\alpha)}\Big(1+\frac{16|z|^2}{(1+r)^2}\Big)^{-\frac{(n+m+\alpha)}{2}}r^{n-1}\,dr\,dz\\
&=\frac{\omega_{2n}\omega_m}{2}\int_{0}^{\infty}\int_0^{\infty}(1+r)^{-(n+m-j+\alpha)}\Big(1+\frac{16s^2}{(1+r)^2}\Big)^{-\frac{(n+m+\alpha)}{2}}r^{n-1}s^{m-1}\,dr\,ds\\
&=\frac{\omega_{2n}\omega_m}{2 \cdot4^m}\int_{0}^{\infty}(1+r)^{-(n-j+\alpha)}r^{n-1}\,dr\int_0^{\infty}(1+s^2)^{-\frac{(n+m+\alpha)}{2}}s^{m-1}\,ds\\
&=\frac{\omega_{2n}\omega_m}{ 4^{m+1}}\int_{0}^{\infty}(1+r)^{-(n-j+\alpha)}r^{n-1}\,dr\int_0^{\infty}(1+s)^{-\frac{(n+m+\alpha)}{2}}s^{m/2-1}\,ds\\
&=\frac{\omega_{2n}\omega_m}{ 4^{m+1}}\frac{\Gamma(n)\Gamma(\alpha-j)}{\Gamma(n-j+\alpha)}\frac{\Gamma\big(\frac{m}{2}\big)\Gamma\big(\frac{n+\alpha}{2}\big)}{\Gamma\big(\frac{n+m+\alpha}{2}\big)}=\frac{\pi^{n+m/2}}{4^m}\frac{\Gamma(\alpha-j)\Gamma(\frac{n+\alpha}{2})}{\Gamma(n-j+\alpha)\Gamma\big(\frac{n+m+\alpha}{2}\big)}.
\end{align*}

\end{proof}

As an immediate corollary, we obtain the following explicit formula for  the kernel $\Phi_{s,\rho}(v,z)$. 
\begin{cor}
\label{cor:explicitP}
The kernel $\Phi_{s,\rho}(v,z)$ can be expressed as
$$
\Phi_{s,\rho}(v,z)=C_1(n,m,s)\rho^{2s} \varphi_{s,\rho}(v,z),
$$
where $\varphi_{s,\rho}(v,z) $ is as in \eqref{eq:varfisr} and $C_1(n,m,s)$ is the constant in \eqref{eq:C1}.
\end{cor}
\begin{proof}
Up to the constant $C_1(n,m,s)$, the expression can be immediately inferred from Theorems \ref{thm:STf} and \ref{thm:Phi1}. To compute $C_1(n,m,s)$ observe that, since $\|\Phi_{s,1}\|_1=1$, we have that $C_1(n,m,s)^{-1}=\int_N\varphi_{s,1}(v,z)\,dv\,dz$, which is given by
$$
\Big(\int_N\big((1+|v|^2)^2 +16 |z|^2\big)^{-\frac{n+m+s}{2}}\,dv\,dz\Big)^{-1}=\frac{4^m}{\pi^{n+m/2}}\frac{\Gamma(n+s)\Gamma\big(\frac{n+m+s}{2}\big)}{\Gamma(s)\Gamma(\frac{n+s}{2})},
$$
by Lemma \ref{lem:I} with $j=0$ and $\alpha=s$.
\end{proof}

Since  $\Phi_{s,\rho}(v,z)$ is an approximate identity, we obtain the $ L^p $ convergence \eqref{eq:convergfT} of Theorem~\ref{thm:sol}.
We are now interested in proving  the assertion about the limit in \eqref{eq:limits}, i.e., that 
$$ 
-\rho^{1-2s} \partial_\rho (\rho^{2s}f \ast \varphi_{s,\rho}) \rightarrow 2^{1-2s}\frac{\Gamma(1-s)}{\Gamma(s)} \mathcal{L}_s f\qquad \text{as } \rho\to 0.
$$ 
First we need to recall the following identity, that was essentially proved by Cowling and Haagerup in \cite[Section 3]{CH}, and it was also shown in \cite[Theorem 3.1]{RT} by a different method (we point out that there is a misprint in the statement of our result \cite[Theorem 3.1]{RT}: the restriction on $s$ should be $0<s<n+1$). We state it here in the most convenient way for our purposes.
\begin{thm}
\label{thm:CH}
Let $\rho>0$ and $0<s<n+1$. Then 
$$
\mathcal{L}_s\varphi_{-s,\rho}(v,z)=C_2(n,m,s)\rho^{2s}\varphi_{s,\rho}(v,z),
$$
where 
$$
C_2(n,m,s)=4^{2s}\frac{\Gamma\big(\frac{n+1+s}{2}\big)}{\Gamma\big(\frac{n+1-s}{2}\big)}\frac{\Gamma\big(\frac{n+m+s}{2}\big)}{\Gamma\big(\frac{n+m-s}{2}\big)}.
$$
\end{thm}
\begin{proof}
Just repeat the proof of \cite[Theorem 3.1]{RT} in the case $H$ type groups, and change $\rho$ into $\rho^2/4$ (there, $\rho$ corresponds to $\delta$).
\end{proof}

\begin{rem}
\label{rem:solution}
In view of the results above, we can write the solution to \eqref{eq:ep} as 
$$
u(v,z,\rho)=C_1(n,m,s)C_2(n,m,s)^{-1}\mathcal{L}_sf\ast \varphi_{-s,\rho}(v,z), 
$$
and let us recall that $\varphi_{s,\rho}(v,z)$ is given in \eqref{eq:varfisr}.
\end{rem}
Therefore, 
$$
-\rho^{1-2s}\partial_{\rho}u=C_1(n,m,s)C_2(n,m,s)^{-1}\mathcal{L}_sf\ast K_{s,\rho},
$$
where $K_{s,\rho}(v,z)=-\rho^{1-2s}\partial_{\rho}\varphi_{-s,\rho}(v,z)$.
Observe that 
$$
-\partial_{\rho}\varphi_{-s,\rho}(v,z)=2\rho(n+m-s)(\rho^2+|v|^2)\big((\rho^2+|v|^2)^2+16|z|^2\big)^{-\frac{n+m-s}{2}-1}.
$$
We can rewrite this as
$$
K_{s,\rho}(v,z)=2(n+m-s)\rho^{2(1-s)}\frac{(\rho^2+|v|^2)}{\big((\rho^2+|v|^2)^2+16|z|^2\big)^{1/2}}\varphi_{1-s,\rho}(v,z).
$$ 
This shows that $K_{s,\rho}(v,z)=\rho^{-2(n+m)}K_{s,1}(\rho^{-1}v,\rho^{-2}z)$ and hence $K_{s,\rho}(v,z),$ after normalisation, is an approximate identity. Thus,
$$
-\lim_{\rho\to0}\rho^{1-2s}\partial_{\rho}u=C_1(n,m,s)C_2(n,m,s)^{-1}C_3(n,m,s)\mathcal{L}_sf,
$$
with $C_3(n,m,s)=\int_NK_{s,1}(v,z)\,dv\,dz$. Then, by Lemma \ref{lem:I} with $j=1$ and $\alpha=2-s$, we have
\begin{equation}
\label{eq:C3}
C_3(n,m,s)=2(n+m-s)\frac{\pi^{n+m/2}}{4^m}\frac{\Gamma(1-s)\Gamma(\frac{n+2-s}{2})}{\Gamma(n-s+1)\Gamma\big(\frac{n+m+2-s}{2}\big)}.
\end{equation}
Finally, collecting \eqref{eq:C1}, \eqref{eq:C2} and \eqref{eq:C3}, we get
\begin{equation}
\label{eq:limro}
-\lim_{\rho\to0}\rho^{1-2s}\partial_{\rho}u=2^{1-2s}\frac{\Gamma(1-s)}{\Gamma(s)}\mathcal{L}_sf.
\end{equation}

From \eqref{eq:convo} in the proof of Theorem \ref{thm:STf}, we see that the solution of the extension problem can be written as $u=f\ast \Phi_{s,\rho}$, and $\Phi_{s,\rho}$ has an explicit expression, see Corollary \ref{cor:explicitP} (and also $\|\Phi_{s,\rho}\|_1=1$). With this we have
\begin{align}
\label{eq:iden2}
 u(v,z,\rho)-f(v,z)&=f\ast  \Phi_{s,\rho}(v,z)-f(v,z)=\int_N\big(f(x)-f(y)\big)\Phi_{s,\rho}(y^{-1}x)\,dy\\
\notag&=C_1(n,m,s)\rho^{2s}\int_N\big(f(x)-f(y)\big)\varphi_{s,\rho}(y^{-1}x)\,dy.
\end{align}
From here, by taking derivative in $\rho$ and multiplying by $\rho^{1-2s}$, we obtain
\begin{align*}
\rho^{1-2s}\partial_{\rho}u(x,\rho)&=\rho^{1-2s}\partial_{\rho}[u(x,\rho)-f(x)]=-C_1(n,m,s)2s\int_N\big(f(x)-f(y)\big)\varphi_{s,\rho}(y^{-1}x)\,dy\\
&\qquad -C_1(n,m,s)\rho\int_N\big(f(x)-f(y)\big)\partial_{\rho}\varphi_{s,\rho}(y^{-1}x)\,dy.
\end{align*}
Now we let $\rho$ tend to $0^+$ to get
\begin{multline*}
-\lim_{\rho\to0^+}\rho^{1-2s}\partial_{\rho}u(v,z,\rho)=C_1(n,m,s) 2s\int_N\frac{\big(f(x)-f(y)\big)}{|xy^{-1}|^{Q+2s}}\,dy
\\
-\lim_{\rho\to0^+}C_1(n,m,s)\rho\int_N\big(f(x)-f(y)\big)\partial_{\rho}\varphi_{s,\rho}(y^{-1}x)\,dy.
\end{multline*}
Assume for a while that we can prove 
\begin{equation}
\label{eq:zero}
\lim_{\rho\to0^+}\rho\int_N\big(f(x)-f(y)\big)\partial_{\rho}\varphi_{s,\rho}(y^{-1}x)\,dy=0.
\end{equation}
Then, in view of \eqref{eq:limro}, we conclude that 
$$
\mathcal{L}_sf(v,z)=\frac{4^m 2^{2s}}{\pi^{n+m/2}}\frac{\Gamma(n+s)\Gamma\big(\frac{n+m+s}{2}\big)}{\Gamma\big(\frac{n+s}{2}\big)|\Gamma(-s)|}\int_N\frac{(f(x)-f(y))}{|xy^{-1}|^{Q+2s}}\,dy,
$$
i.e., the identity \eqref{eq:ir} in Theorem \ref{thm:sol}.

It remains to prove \eqref{eq:zero}. For that, we need the following Mean Value Theorem adapted to $H$-type groups, see \cite[(1.33)]{FS}.
\begin{thm}[Folland-Stein]
\label{thm:FS}
There exist constants $C>0$ and $\beta>0$ such that for all $f\in C^1(N)$  we have
$$
|f(yx)-f(x)|\le C(|y|+|y|^2)\sum_{j=1}^{2n+m}\sup_{|y'|\le \beta|y|}|Y_jf(y'x)|,
$$
where $Y_j=X_j$, $j=1,\ldots,2n$ and $Y_j=Z_j$, $j=2n+1\ldots,2n+m$. 
\end{thm}
Note that by a change of variables we have
$$
\rho\int_N\big(f(x)-f(y)\big)\partial_{\rho}\varphi_{s,\rho}(y^{-1}x)\,dy=\rho\int_N\big(f(x)-f(xy^{-1})\big)\partial_{\rho}\varphi_{s,\rho}(y)\,dy.
$$
A simple calculation shows  that
$$
\partial_{\rho}\varphi_{s,\rho}(y)=\rho^{-1-2s}\rho^{-2(n+m)}\Psi(\delta_{\rho}(y)),
$$
where 
$$
\Psi(v,z)=\frac{(1+|v|^2)}{((1+|v|^2)^2+|z|^2)^{1/2}}((1+|v|^2)^2+|z|^2)^{-\frac{n+m+1+s}{2}},
$$ 
and $\Psi(\delta_{\rho}(y))=\Psi(\delta_{\rho}(v,z))=\Psi(\rho^{-1}v,\rho^{-2}z)$. Then, by Theorem \ref{thm:FS} we have
$$
\rho\int_N\big(f(x)-f(xy^{-1})\big)|\partial_{\rho}\varphi_{s,\rho}(y)|\,dy
\le C \rho^{-2s}\rho^{-2(n+m)}\int_N\big (|y|+|y|^2)\Psi(\delta_{\rho}(y))\,dy.
$$
By a change of variable, the above is bounded by $C \rho^{1-2s}\int_N\big (|y|+\rho|y|^2)\Psi(y)\,dy$, which tends to $0$ as $\rho\to 0^+$, for $0<s<1/2$, since $ (|y|+|y|^2)\Psi(y)$ is integrable. 

Finally, we can now compute the limit in \eqref{eq:limit2}. From \eqref{eq:iden2}, by taking limits, we obtain
\begin{align*}
\lim_{\rho\to 0^+} \rho^{-2s}(u(v,z,\rho)-f(v,z))&=C_1(m,n,s)\int_{N}\frac{(f(x)-f(y))}{|xy^{-1}|^{Q+2s}}\,dy\\
&=\frac{|\Gamma(-s)|}{4^s\Gamma(s)} \mathcal{L}_sf(v,z),
\end{align*}
where the second equality follows from the identity \eqref{eq:ir} just proven.
With this, the proof of Theorem \ref{thm:sol} is complete.

\subsection{Isometry property of the solution operator for the extension problem}
\label{subsec:iso}
In this subsection we will prove Theorem \ref{thm:MOZgen}. In order to prove this theorem, first we prove the isometry property of the solution operator $P_s$ (see the Introduction) for the extension problem in the Heisenberg group context given by M\"ollers et al in \cite{MOZ}. We will modify slightly the notation here and use $(z,t)\in \C^n\times \R$ for elements of the Heisenberg group $\H^n$.

The solution of the extension problem associated to the sublaplacian $\mathcal{L}$ on $\H^n$ is given by 
$$
u(z,t,\rho)=C_1(n,s)\rho^{2s}f\ast \varphi_{s,\rho}(z,t)
$$
where $C_1(n,s):=C_1(n,1,s)$ and
$$
\varphi_{s,\rho}(z,t)= \big((\rho^2+|z|^2)^2 +16 t^2\big)^{-\frac{n+1+s}{2}}.
$$
We write the solution operator $f\to u$ as $P_sf$. Thus $P_sf=C_1(n,s)\rho^{2s}f\ast \varphi_{s,\rho}(z,t)$. Note that the solution can be considered as a function on the $(2n+3)$-dimensional Heisenberg group $\H^{n+1}=\C^{n+1}\times \R$ which is radial in the extra variable $\zeta$. On $\H^{n+1}$ we use the coordinates $(z,\zeta,t)$. Thus
$$
P_sf(z,\zeta,t)=C_1(n,s)|\zeta|^{2s}f\ast \varphi_{s,|\zeta|}(z,t).
$$
Taking this point of view, M\"ollers et al have considered $f\to P_sf$ as an extension map taking functions $f$ on $\H^n$ into functions on $\H^{n+1}$. They have proved that $P_s$ defines an isometry between certain Sobolev spaces defined on $\H^n$ and $\H^{n+1}$.

To motivate the proof let us consider the case of the Laplacian on $\R^n$. In view of the Stinga-Torrea formula  (3.3) the solution of the extension problem associated  to $\Delta$  
\begin{equation}
\label{eq:epEu}
\big(\Delta+\partial_{\rho}^2+\frac{1-s}{\rho}\partial_{\rho}\big)u=0, \qquad u(x,0)=f(x)
\end{equation}
is given by ( see  \cite{ST} and also \cite{BRT})
$$
u(x,\rho)=c_{n,s} \rho^s f\ast u_{s,\rho}(x)
$$
where now $u_{s,\rho}(x)=(\rho^2+|x|^2)^{-\frac{n+s}{2}}$ and $c_n$ is given by 
$$
c_n^{-1}=\int_{\R^n}(1+|x|^2)^{-\frac{n+s}{2}}\,dx.
$$
Let $P_s$ be defined by $P_sf(x,y)=u(x,|y|)$, $(x,y)\in \R^n\times \R$. Let $H^s(\R^n)$ stand for the homogeneous Sobolev space defined to be the set of all distributions $f$ such that $(-\Delta)^{s/2}f\in L^2(\R^n)$. Equivalently, $f\in H^s(\R^n)$ if and only if $\widehat{f}\in L^2(\R^n, |\xi|^{2s}\,d\xi)$. In \cite{MOZ} the following theorem has been proved, and here we provide a different proof.
\begin{thm}
For $0<s<n$, $P_s:H^{s/2}(\R^n)\to H^{(s+1)/2}(\R^{n+1})$ is a constant multiple of an isometry.
\end{thm}
\begin{proof}
In proving this theorem, we make use of the identity (see \cite[Lemma 2.2]{BRT}, also \cite{ST})
\begin{equation}
\label{eq:po}
\rho^sf\ast u_{s,\rho}(x)=c_{n,s}(-\Delta)^{s/2}f\ast u_{-s,\rho}(x)
\end{equation}
where $c_{n,s}$ is an explicit constant. For a function $\varphi(x,y)$ on $\R^{n+1}$ note that 
$$
\widehat{\varphi}(\xi,\eta)=\int_{\R}e^{-iy \eta}\Big(\int_{\R^n}\varphi(x,y)e^{-ix\cdot \xi}\,dx\Big)\,dy.
$$
Consequently, taking into account of \eqref{eq:po},
$$
(P_sf)^{\widehat{}}(\xi,\eta)=c_{n,s}\Big(\int_{\R^{n+1}}(y^2+|x|^2)^{-\frac{n-s}{2}}e^{-i(x\cdot \xi+y\eta)}\,dx\,dy\Big)\big((-\Delta)^{s/2}f\big)^{\widehat{}}(\xi).
$$
Since $(y^2+|x|^2)^{-\frac{n+1-(1+s)}{2}}$ is homogeneous of degree $-(n+1)-(1+s)$ on $\R^{n+1}$, its Fourier transform is a constant multiple of $(\eta^2+|\xi|^2)^{-\frac{s+1}{2}}$. Thus
$$
(P_sf)^{\widehat{}}(\xi,\eta)=c_{n,s}(\eta^2+|\xi|^2)^{-\frac{s+1}{2}}|\xi|^s\widehat{f}(\xi),
$$
where $c_{n,s}$ is an explicit constant. In view of the above calculations,
$$
\|P_sf\|_{H^{(s+1)/2}(\R^{n+1})}^2=c_{n,s}^2\int_{\R^n}\int_{\R}(\eta^2+|\xi|^2)^{-\frac{s+1}{2}}|\xi|^{2s}|\widehat{f}(\xi)|^2\,d\eta\,d\xi,
$$
which after integration in the $\eta$ variable gives 
$$
\|P_sf\|_{H^{(s+1)/2}(\R^{n+1})}^2=c_{n,s}^2\int_{\R^{n}}|\xi|^{s}|\widehat{f}(\xi)|^2\,d\eta\,d\xi=c_{n,s}^2\|f\|_{H^{s/2}(\R^{n})}^2.
$$
This proves the theorem.
\end{proof}

Analogously as in the context of $H$-type groups, we introduce the following variant of a Sobolev space. Let $\widetilde{W}_0^{s,2}(\R^n)$ be the completion of $C_0^{\infty}(\R^n\times \R)$ with respect to the norm 
$$
\|u\|_{(s)}^2=\int_0^{\infty}\int_{\R^n}|\nabla u(x,\rho)|^2\rho^{1-s}\,dx\,d\rho.
$$ 
It can be easily checked that this is indeed a norm.
We have the following.
\begin{prop}
\label{prop:1}
For $ 0 < s < 1, $ let $ f$ be a real valued function in $  H^{s/2}(\R^n)$, i.e.,  $ f$, $ (-\Delta)^{s/4}f\in L^2(\R^n)$.  If $ u $ is the solution of the extension problem \eqref{eq:epEu} with initial condition $ f$, then
$$
\int_0^{\infty}\int_{\R^n}|\nabla u(x,\rho)|^2\rho^{1-s}\,dx\,d\rho=C\| f\|_{H^{s/2}(\R^{n})}^2.
$$
\end{prop}
\begin{proof}
In what follows, we make use of the following fact which can be easily proved. If $ f $ and $ g $ are from $ L^2(\R^n) $ then their convolution $ f*g $ is uniformly continuous and vanishes at infinity. This can be proved by approximating $ f $ and $ g $ by sequence of compactly supported smooth functions. Since $ f $ and $ \rho^{s} u_{s,\rho} $ are from $ L^2(\R^n) $ it follows that the solution $ u $ of the extension problem vanishes at infinity as a function of $ x $ for any fixed $ \rho.$ Moreover, $ \partial_{x_j} u_{s,\rho} \in L^2(\R^n) $. The same is true of $ \partial_\rho u_{s,\rho}$.  Hence $ \partial_{x_j} u $ also vanishes at infinity. Integrating by parts and making use of the above, we see that
$$  
\int_{\R^n} |\partial_{x_j}u(x,\rho)|^2 \,dx= \int_{\R^n} u(x,\rho) \partial_{x_j}^2u(x,\rho) \,dx.
$$
We also observe that $ |u(x,\rho)| \leq \rho^{s} \|f\|_2 \| u_{s,\rho}\|_2 $ which after a simple calculation shows that $ u(x,\rho) $ goes to $ 0 $ as $ \rho $ tends to infinity. The same is true of $ \partial_\rho u(x,\rho).$
 In view of these, a similar calculation with the $\rho$-derivative gives
$$
\int_0^{\infty}(\partial_{\rho} u(x,\rho))^2 \rho^{1-s} d\rho   =\int_0^{\infty} u(x,\rho) \pa_\rho\big( \rho^{1-s}\pa_\rho u(x,\rho)\big)\,d\rho- u(x, 0)\,\lim_{\rho \rightarrow 0}\big(\rho^{1-s}\pa_{\rho} u\big)(x, \rho).
$$
Adding them up and recalling that $ u $ solves the extension problem with initial condition $ f$ we obtain the result.
\end{proof}
\begin{prop}
\label{prop:2}
Under the same hypotheses as in Proposition \ref{prop:1}, we have
$$
\int_0^{\infty}\int_{\R^n}u(x,\rho)^2\rho^{-1-s}\,dx\,d\rho= C\| f\|_{H^{s/2}(\R^{n})}^2.
$$
\end{prop}
\begin{proof}
Observe that 
$$
\widehat{u}(\xi,\rho)=c_{n,s}\rho^{s}\widehat{f}(\xi)\int_{\R^n}(\rho^2+|x|^2)^{-\frac{n+s}{2}}e^{-i x\cdot \xi}\,dx=c_{n,s}\widehat{f}(\xi)G(\rho|\xi|),
$$
where $ G $ is defined by  $$G(\xi):=\int_{\R^n}(1+|x|^2)^{-\frac{n+s}{2}}e^{-i x\cdot \xi}\,dx.$$ The asymptotic properties of $ G $ are well known, see \cite[p. 132, (29), (30)]{StSingular}. As $ G(\rho) $ behaves like $ \rho^s $ near $ 0 $ it follows that $ \int_0^{\infty}G(\rho)^2\rho^{-1-s}\,d\rho < \infty $ and hence
\begin{multline*}
\int_0^{\infty}\int_{\R^n}|\widehat{f}(\xi)|^2G(\rho|\xi|)^2\rho^{-1-s}\,d\rho\,d\xi\\
=\int_{\R^n}|\widehat{f}(\xi)|^2|\xi|^s\Big(\int_0^{\infty}G(\rho)^2\rho^{-1-s}\,d\rho\Big)\,d\xi=C\| f\|_{H^{s/2}(\R^{n})}^2.
\end{multline*}
\end{proof}
\begin{prop}
\label{prop:Ws2}
For $ 0 < s < 1, $ let $ f $ be a real valued function in $  H^s(\R^n)$.  If $ u $ is the solution of the extension problem \eqref{eq:epEu} with initial condition $ f$, then 
  $ u \in \widetilde{W}_0^{s,2}(\R^n)$. 
\end{prop}
\begin{proof}
We know that the function 
$$
u=c_{n,s}f\ast \rho^{s}(\rho^2+|x|^2)^{-\frac{n+s}{2}}
$$
solves the extension problem \eqref{eq:epEu} and 
$$
\int_0^{\infty}\int_{\R^n}|\nabla u(x,\rho)|^2\rho^{1-s}\,dx\,d\rho<\infty.
$$
Then, we have to approximate $u$ by $C_0^{\infty}$ functions on $\R^n\times \R$. Let $\eta\in C_0^{\infty}(\R)$ be such that $\eta(t)=0$ for $|t|>2$ and $\eta(t)=1$ for $|t|\le 1$. Consider 
$$
u_j(x,\rho)=\eta\big(2^{-j}(\rho^2+|x|^2)\big)u(x,\rho)=:\psi_j(x,\rho)u(x,\rho),
$$
after extending $u$ to $\R^n\times \R$ as an even function in $\rho$. Then, $u_j$ is compactly supported and 
$$
\nabla u_j=\nabla \psi_j u+\psi_j\nabla u
$$
so that
$$
\nabla(u-u_j)=\nabla \psi_j u+(1-\psi_j)\nabla u.
$$
Note that $1-\psi_j=0$ on $(\rho^2+|x|^2)\le 2^j$. Hence
$$
\int_0^{\infty}\int_{\R^n}|(1-\psi_j)\nabla u|^2\rho^{1-s}\,dx\,d\rho\le \int_{(\rho^2+|x|^2)>2^j}|\nabla u|^2\rho^{1-s}\,dx\,d\rho\to 0 \quad \text{as} \quad j\to \infty.
$$
On the other hand, as $\eta'=0$ on $(\rho^2+|x|^2)\le 2^j$ and $(\rho^2+|x|^2)\ge 2^{j+1}$, 
$$
\big|\frac{\partial}{\partial x_j}\psi_j(x,\rho)\big|=2^{-j}2|x_j|\eta'\big(2^{-j}(\rho^2+|x|^2)\big)\le 2^{-j}(\rho^2+|x|^2)^{1/2}\eta'\big(2^{-j}(\rho^2+|x|^2)\big)\le C\rho^{-1}.
$$
Similarly, $|\partial_{\rho}\psi_j(x,\rho)|\le C\rho^{-1}$. Therefore,
\begin{align*}
\int_0^\infty \int_{\R^n} |\nabla \psi_j(x,\rho)|^2 |u|^2 \rho^{1-s}\,dx\, d\rho 
&\leq C  \int_{ \rho^2+|x|^2 \geq 2^j } | u|^2 \rho^{-1-s}\,dx\, d\rho\\
&\le C \int_0^{\infty}\int_{\R^n} | u|^2 \rho^{-1-s}\,dx\, d\rho=\|(-\Delta)^{s/4}f\|_2<\infty,
\end{align*}
by Proposition \ref{prop:2} and the hypothesis on $f$.
Hence $\int_{ \rho^2+|x|^2 \geq 2^j } | u|^2 \rho^{-1-s}\,dx\, d\rho\to 0$ as $j\to \infty$, i.e., 
$$
\int_0^\infty \int_{\R^n} |\nabla \psi_j(x,\rho)|^2 |u|^2 \rho^{1-s}\,dx\, d\rho \to 0 \quad \text{as} \quad j\to \infty.
$$
Thus, $u$ is approximated by $u_j$, and the proof is complete. 
\end{proof}

We now turn to the case of the Heisenberg group. Here the Sobolev spaces $H^s(\H^n)$ are defined via the relation $f\in H^s(\H^n)$ if and only if $\mathcal{L}_s^{1/2}f\in L^2(\H^n)$ where $\mathcal{L}_s$ is the conformally invariant fractional power of $\mathcal{L}$. In view of the Plancherel theorem for the Heisenberg group, we have
$$
\|f\|_{H^s(\H^n)}^2=(2\pi)^{-n-1}\int_{-\infty}^{\infty}\|(\mathcal{L}_s^{1/2}f)^{\widehat{}}(\lambda)\|_{\operatorname{HS}}^2|\lambda|^n\,d\lambda.
$$
Recalling the definition of $\mathcal{L}_s$ we note that
$$
(\mathcal{L}_s^{1/2}f)^{\widehat{}}(\lambda)=\widehat{f}(\lambda)\Big(\frac{\Gamma\big(\frac{H_n(\lambda)}{2|\lambda|}+\frac{1+s}{2}\big)}{\Gamma\big(\frac{H_n(\lambda)}{2|\lambda|}+\frac{1-s}{2}\big)}\Big)^{1/2}(2|\lambda|)^{s/2}
$$
where $H_n(\lambda)$ stands for the scaled Hermite operator $-\Delta+\lambda^2|x|^2$ on $\R^n$ (we changed a little bit the notation from Subsection \ref{subsec:subl}). The Hilbert--Schmidt norm can be calculated in terms of the Hermite basis $\Phi_{\alpha}^{\lambda}(x)$, $\alpha\in \Na^n$. As they are eigenfunctions of the Hermite operator $H_n(\lambda)$ with eigenvalues $(2|\alpha|+n)|\lambda|$, we see that
$$
\|f\|_{H^s(\H^n)}^2=2^s(2\pi)^{-n-1}\int_{-\infty}^{\infty}\Big(\sum_{\alpha\in \Na^n}\frac{\Gamma\big(\frac{2|\alpha|+n+1+s}{2}\big)}{\Gamma\big(\frac{2|\alpha|+n+1-s}{2}\big)}\|\widehat{f}(\lambda)\Phi_{\alpha}^{\lambda}\|_2^2\Big)|\lambda|^{n+s}\,d\lambda.
$$

In calculating the norm of a function $u\in H^s(\H^{n+1})$ we make use of the orthonormal basis $\Phi_{\alpha,j}^{\lambda}(x,y)=\Phi_{\alpha}^{\lambda}(x)h_j^{\lambda}(y)$, $\alpha\in \Na^n$, $j\in \Na$ where $h_j^{\lambda}(y) = \Phi_j^\lambda(y)$ are the one dimensional Hermite functions forming an orthonormal basis for $L^2(\R)$. Thus
$$
\|u\|_{H^{s+1}(\H^{n+1})}^2=2^{s+1}(2\pi)^{-n-2}\int_{-\infty}^{\infty}\Big(\sum_{(\alpha,j)\in \Na^n\times \Na}\frac{\Gamma\big(\frac{2|\alpha|+(n+1)+1+(1+s)}{2}\big)}{\Gamma\big(\frac{2|\alpha|+(n+1)+1-(1+s)}{2}\big)}\|\widehat{u}(\lambda)\Phi_{\alpha,j}^{\lambda}\|_2^2\Big)|\lambda|^{(n+s+2)}\,d\lambda.
$$

We are now ready to prove the folllowing result. Let $P_sf(z,\zeta,t)=C_1(n,s)|\zeta|^{2s}f\ast \varphi_{s,|\zeta|}(z,t)$ be the solution operator. 
\begin{thm}
\label{thm:MOZ}
For $0<s<n+1,$  the solution operator $P_s:H^{s}(\H^n)\to H^{s+1}(\H^{n+1})$ satisfies the isometry property
$$
\|u\|_{H^{s+1}(\H^{n+1})}^2=\frac{\pi^{2n+4}\Gamma(s)\Gamma(1+s)}{2^{2n-2-6s}\Gamma\big(\frac{n+1-s}{2}\big)^4}\|f\|_{H^{s}(\H^{n})}^2.
$$
\end{thm}
\begin{proof}
We begin with the observation that the (group) Fourier transform of a function $u$ on $\H^{n+1}$ is defined by
$$
\widehat{u}(\lambda)=\int_{\H^{n+1}}u(z,\zeta,t)\pi_{\lambda}(z,\zeta,t)\,dz\,d\zeta\,dt.
$$
Here $\pi_{\lambda}(z,\zeta,t)$ are the Schr\"odinger representations realised on $L^2(\R^{n+1})$. Since 
$$
(z,\zeta,t)=(0,\zeta,0)(z,0,t), \quad \pi_{\lambda}(z,\zeta,t)=\pi_{\lambda}(0,\zeta,0)\pi_{\lambda}(z,0,t)
$$
and hence
$$
\widehat{u}(\lambda)=\int_{\C}\pi_{\lambda}(0,\zeta,0)\Big(\int_{\H^{n}}u(z,\zeta,t)\pi_{\lambda}(z,0,t)\,dz\,dt\Big)\,d\zeta.
$$
Since the solution $u(z,\zeta,t)=P_sf(z,\zeta,t)$ is given by convolution of $f$ with a function, the inner integral simplifies. 

At this point we make use of Theorem \ref{thm:CH} with $m=1$, namely,
$$
\mathcal{L}_s\varphi_{-s,\rho}(z,t)=C_2(n,s)\rho^{2s}\varphi_{s,\rho}(z,t),
$$
where $C_2(n,s):=C_2(n,1,s)$. In view of this we can write the solution as 
$$
u(z,\zeta,t)=C_1(n,s)C_2(n,s)^{-1}\mathcal{L}_sf\ast \varphi_{-s,|\zeta|}(z,t).
$$
Thus
$$
\int_{\H^n}u(z,\zeta,t)\pi_{\lambda}(z,0,t)\,dz\,dt=C_1(n,s)C_2(n,s)^{-1}(\mathcal{L}_sf)^{\widehat{}}(\lambda)\int_{\H^n}\varphi_{-s,|\zeta|}(z,t)\pi_{\lambda}(z,0,t)\,dz\,dt,
$$
where $(\mathcal{L}_sf)^{\widehat{}}(\lambda)$ is the Fourier transform of $\mathcal{L}_sf$ on $\H^n$ and hence acts on $L^2(\R^n)$. From this we obtain 
\begin{equation}
\label{eq:uw}
\widehat{u}(\lambda)=C_1(n,s)C_2(n,s)^{-1}(\mathcal{L}_sf)^{\widehat{}}(\lambda)\int_{\H^{n+1}}\big((|z|^2+|\zeta|^2)^2+16t^2\big)^{-\frac{n+1-s}{2}}\pi_{\lambda}(z,\zeta,t)\,dz\,d\zeta\,dt.
\end{equation}
For $g=(z,\zeta,t)\in \H^n$, let $|g|=\big((|z|^2+|\zeta|^2)^2+16t^2\big)^{1/4}$ be the homogeneous norm. Then \eqref{eq:uw} can be rewritten as
\begin{equation}
\label{eq:uwr}
\widehat{u}(\lambda)=C_1(n,s)C_2(n,s)^{-1}(\mathcal{L}_sf)^{\widehat{}}(\lambda)\int_{\H^{n+1}}|g|^{-2(n+2)+2(s+1)}\pi_{\lambda}(g)\,dg.
\end{equation}
It is known, see \cite[p. 127]{RT}, that $C_{n+1,s+1}^{-1}|g|^{-2(n+2)+2(1+s)}$ is a fundamental solution of the operator $\mathcal{L}_{s+1}$ on $\H^{n+1}$, if we take
\begin{equation*}
C_{n,s}=\frac{\pi^{n+1}\Gamma(s)}{2^{n+1-3s}\Gamma\big(\frac{n+1-s}{2}\big)^2}.
\end{equation*}
Consequently,
$$
\int_{\H^{n+1}}|g|^{-2(n+2)+2(s+1)}\pi_{\lambda}(g)\,dg= C_{n+1,s+1}(2|\lambda|)^{-s-1}\frac{\Gamma\big(\frac{H_{n+1}(\lambda)}{2|\lambda|}+\frac{1-(1+s)}{2}\big)}{\Gamma\big(\frac{H_{n+1}(\lambda)}{2|\lambda|}+\frac{1+(1+s)}{2}\big)}.
$$

Therefore, in view of this and \eqref{eq:uwr} we see that 
$$
\widehat{u}(\lambda)\Phi_{\alpha,j}^{\lambda}(x,y)=C_{n+1,s+1}(2|\lambda|)^{-s-1}\frac{\Gamma\big(\frac{2|\alpha|+2j+(n+1)+1-(1+s)}{2}\big)}{\Gamma\big(\frac{2|\alpha|+2j+(n+1)+1+(1+s)}{2}\big)}(\mathcal{L}_sf)^{\widehat{}}(\lambda)\Phi_{\alpha,j}^{\lambda}(x,y).
$$
As $(\mathcal{L}_sf)^{\widehat{}}(\lambda)=\widehat{f}(\lambda)\frac{\Gamma\big(\frac{H_{n}(\lambda)}{2|\lambda|}+\frac{1+s)}{2}\big)}{\Gamma\big(\frac{H_{n}(\lambda)}{2|\lambda|}+\frac{1-s)}{2}\big)}(2|\lambda|)^{s}$ acts only on $\Phi_{\alpha}^{\lambda}$, we get
$$
\widehat{u}(\lambda)\Phi_{\alpha,j}^{\lambda}(x,y)=C_{n+1,s+1}^2|\lambda|^{-1}\frac{\Gamma\big(\frac{2|\alpha|+2j+(n+1)+1-(1+s)}{2}\big)}{\Gamma\big(\frac{2|\alpha|+2j+(n+1)+1+(1+s)}{2}\big)}\frac{\Gamma\big(\frac{2|\alpha|+n+1+s}{2}\big)}{\Gamma\big(\frac{2|\alpha|+n+1-s)}{2}\big)}\widehat{f}(\lambda)\Phi_{\alpha}^{\lambda}(x)h_j^{\lambda}(y),
$$
so that 
$$
\|\widehat{u}(\lambda)\Phi_{\alpha,j}^{\lambda}\|_2^2=C_{n+1,s+1}^2|\lambda|^{-2}\frac{\Gamma\big(\frac{2|\alpha|+2j+(n+1)+1-(1+s)}{2}\big)^2}{\Gamma\big(\frac{2|\alpha|+2j+(n+1)+1+(1+s)}{2}\big)^2}\frac{\Gamma\big(\frac{2|\alpha|+n+1+s}{2}\big)^2}{\Gamma\big(\frac{2|\alpha|+n+1-s)}{2}\big)^2}\|\widehat{f}(\lambda)\Phi_{\alpha}^{\lambda}\|_2^2.
$$
Recalling the expression for the norm of $u$ in $H^{s+1}(\H^{n+1})$ we have to multiply the above with
$$
\frac{\Gamma\big(\frac{2|\alpha|+2j+(n+1)+1+(1+s)}{2}\big)}{\Gamma\big(\frac{2|\alpha|+2j+(n+1)+1-(1+s)}{2}\big)}
$$
and sum over $\alpha$ and $j$:
$$
\sum_{j=0}^{\infty}\frac{\Gamma\big(\frac{2|\alpha|+2j+(n+1)+1+(1+s)}{2}\big)}{\Gamma\big(\frac{2|\alpha|+2j+(n+1)+1-(1+s)}{2}\big)}\frac{\Gamma\big(\frac{2|\alpha|+2j+(n+1)+1-(1+s)}{2}\big)^2}{\Gamma\big(\frac{2|\alpha|+2j+(n+1)+1+(1+s)}{2}\big)^2} =\sum_{j=0}^{\infty}\frac{\Gamma\big(\frac{2|\alpha|+2j+(n+1)+1-(1+s)}{2}\big)}{\Gamma\big(\frac{2|\alpha|+2j+(n+1)+1+(1+s)}{2}\big)}.
$$
The above can be computed using properties of hypergeometric functions. Recall that 
$$
F(a,b;c;z)=\sum_{s=0}^{\infty}\frac{(a)_s(b)_s}{(c)_ss!}z^s=\sum_{s=0}^{\infty}\frac{\Gamma(a+s)\Gamma(b+s)\Gamma(c)}{\Gamma(a)\Gamma(b)\Gamma(c+s)s!}z^s.
$$
With $a=\frac{2|\alpha|+n+1-s}{2}$, $b=1$, $c=\frac{2|\alpha|+n+3+s}{2}$ and $z=1$ we see that 
$$
\sum_{j=0}^{\infty}\frac{\Gamma\big(\frac{2|\alpha|+2j+(n+1)+1-(1+s)}{2}\big)}{\Gamma\big(\frac{2|\alpha|+2j+(n+1)+1+(1+s)}{2}\big)}
=\sum_{j=0}^{\infty}\frac{\Gamma\big(\frac{2|\alpha|+n+1-s}{2}+j\big)\Gamma(1+j)}{\Gamma\big(\frac{2|\alpha|+n+3+s}{2}+j\big)j!}=\frac{\Gamma(a)\Gamma(1)}{\Gamma(c)}F(a,1;c,1).
$$
Now making use of \cite[15.4.20]{OlMax}:
$$
F(a,b;c;1)=\frac{\Gamma(c)\Gamma(c-a-b)}{\Gamma(c-a)\Gamma(c-b)}, \qquad \Re(c-a-b)>0,
$$ 
we obtain 
$$
\sum_{j=0}^{\infty}\frac{\Gamma\big(\frac{2|\alpha|+2j+(n+1)+1-(1+s)}{2}\big)}{\Gamma\big(\frac{2|\alpha|+2j+(n+1)+1+(1+s)}{2}\big)}=\frac{\Gamma(s)}{\Gamma(1+s)}\frac{\Gamma\big(\frac{2|\alpha|+n+1-s}{2}\big)}{\Gamma\big(\frac{2|\alpha|+n+1+s}{2}\big)}.
$$
Consequently, 
\begin{multline*}
\sum_{\alpha\in \Na^n}\sum_{j=0}^{\infty}\frac{\Gamma\big(\frac{2|\alpha|+2j+(n+1)+1+(1+s)}{2}\big)}{\Gamma\big(\frac{2|\alpha|+2j+(n+1)+1-(1+s)}{2}\big)}\|\widehat{u}(\lambda)\Phi_{\alpha,j}^{\lambda}\|_2^2\\
=C_{n+1,s+1}^2\frac{\Gamma(s)}{\Gamma(1+s)}|\lambda|^{-2}\sum_{\alpha\in \Na^n}\frac{\Gamma\big(\frac{2|\alpha|+n+1+s)}{2}\big)}{\Gamma\big(\frac{2|\alpha|+n+1-s}{2}\big)}\|\widehat{f}(\lambda)\Phi_{\alpha}^{\lambda}\|_2^2.
\end{multline*}
Integrating the above with respect to $|\lambda|^{n+2+s}\,d\lambda$ we obtain
$$
\|u\|_{H^{s+1}(\H^{n+1})}^2=C_{n+1,s+1}^2\frac{\Gamma(s)}{\Gamma(1+s)}\|f\|_{H^{s}(\H^{n})}^2.
$$
\end{proof}

Now we are ready to prove Theorem \ref{thm:MOZgen}.
\begin{proof}[Proof of Theorem  \ref{thm:MOZgen}]
We begin with the observation that if $u$ is a solution of the extension problem for $\mathcal{L}$ on $N$, then for any $\omega\in \mathbb{S}^{m-1}$, $R_{\omega}u$ is a solution of the extension problem for $\mathcal{L}$ on $\H_{\omega}^n$. (Since all $\H_{\omega}^n$ are isomorphic to $\H^{n}$ we use the same notation $\mathcal{L}$ to denote these sublaplacians). This can be easily seen as follows: since $u(v,z,\rho)=C_1(n,m,s)\rho^{2s}f\ast \varphi_{s,\rho}(v,z)$, by taking Radon transform in the $z$-variable and recalling that $R_{\omega}(f\ast \varphi_{s,\rho})(v,t)=R_{\omega}f\ast_{\omega}R_{\omega}\varphi_{s,\rho}(v,t)$, where $\ast_{\omega}$ is the convolution on $\H_{\omega}^n$, we have
$$
R_{\omega}u(v,t,\rho)=C_1(n,m,s)\rho^{2s}R_{\omega}f\ast_{\omega}R_{\omega}\varphi_{s,\rho}(v,t).
$$
We have already calculated $R_{\omega}\varphi_{s,\rho}(v,t)$ in the proof of Theorem \ref{thm:Phi1}. It is given by $ R_\omega\varphi_{s,\rho}(v,t)=\frac{C_1(n,1,s)}{C_1(n,m,s)} \big((\rho^2+|v|^2)^2+16t^2\big)^{-\frac{n+1+s}{2}}$ which is considered as  a kernel on $\H_{\omega}^n\cong \H^n. $ Thus we have
$$
C_1(n,m,s)R_{\omega}\varphi_{s,\rho}(v,t)=C_1(n,1,s) \big((\rho^2+|v|^2)^2+16t^2\big)^{-\frac{n+1+s}{2}}
$$
and it is clear that $R_{\omega}u(v,t,\rho)$ solves the extension problem for $\mathcal{L}$ on $\H_{\omega}^n$ with initial condition $R_{\omega}f$. As $(-\Delta_z)$ commutes with $\mathcal{L}$, it follows that $(-\Delta_z)^{\frac{m-1}{4}}u=C_1(n,m,s)\rho^{2s}(-\Delta_z)^{\frac{m-1}{4}}f\ast \varphi_{s,\rho}$ and consequently, 
$$
R_{\omega}(-\Delta_z)^{\frac{m-1}{4}}u(v,t)=C_1(n,1,s)\rho^{2s}R_{\omega}(-\Delta_z)^{\frac{m-1}{4}}f\ast_{\omega}\widetilde{\varphi}_{s,\rho}(v,t),
$$
where $\widetilde{\varphi}_{s,\rho}(v,t)= \big((\rho^2+|v|^2)^2+16t^2\big)^{-\frac{n+1+s}{2}}$.
Now in view of Theorem \ref{thm:MOZ}, we have 
\begin{align*}
\|R_{\omega}\big((-\Delta_z)^{\frac{m-1}{4}}u\big)\|^2_{H^{s+1}(\H^{n+1})}=C \|R_{\omega}\big((-\Delta_z)^{\frac{m-1}{4}}f\big)\|^2_{H^{s}(\H^{n})}.
\end{align*}
Integrating the above over $\mathbb{S}^{m-1}$ we obtain
$$
\int_{\mathbb{S}^{m-1}}\|R_{\omega}\big((-\Delta_z)^{\frac{m-1}{4}}u\big)\|^2_{H^{s+1}(\H^{n+1})}\,d\sigma(\omega)=C\int_{\mathbb{S}^{m-1}} \|R_{\omega}\big((-\Delta_z)^{\frac{m-1}{4}}f\big)\|^2_{H^{s}(\H^{n})}\,d\sigma(\omega).
$$
Recalling the definition of the norm on $H^s(\H^n)$, the right hand side of the above reads
$$
\int_{\mathbb{S}^{m-1}}\int_{\H^n}|\mathcal{L}_s^{1/2}R_{\omega}\big((-\Delta_z)^{\frac{m-1}{4}}f\big)(v,t)|^2\,dv\,dt\,d\sigma(\omega).
$$
The proof will be completed by showing that the above integral is a constant multiple of $\int_N|\mathcal{L}_s^{1/2}f(v,z)|^2\,dv\,dz$. By Plancherel theorem for the Fourier transform on $\H^n$
\begin{multline*}
\int_{\H^n}|\mathcal{L}_s^{1/2}R_{\omega}\big((-\Delta_z)^{\frac{m-1}{4}}f\big)(v,t)|^2\,dv\,dt\\
=(2\pi)^{-n-1}\int_{-\infty}^{\infty}\Big(\sum_{\alpha\in \Na^n}\frac{\Gamma\big(\frac{2|\alpha|+n+1+s}{2}\big)}{\Gamma\big(\frac{2|\alpha|+n+1-s}{2}\big)}\|\pi_{\lambda}\big(R_{\omega}\big((-\Delta_z)^{\frac{m-1}{4}}f\big)\Phi_{\alpha}^{\lambda}\|_2^2\Big)|\lambda|^{n+s}\,d\lambda
\end{multline*}
where we have used $\pi_{\lambda}(F)$ to stand for the Fourier transform $\widehat{F}(\lambda)$ of a function $F$ on $\H^n$. But we note that for any $g$ on $N$, $\pi_{\lambda}(R_{\omega}g)=\pi_{\lambda,\omega}(g)$ (recall the notation from Subsection~\ref{subsec:reprH}). Using this we have
$$
\pi_{\lambda}(R_{\omega}(-\Delta_z)^{\frac{m-1}{4}}f)=\pi_{\lambda,\omega}((-\Delta_z)^{\frac{m-1}{4}}f)=|\lambda|^{\frac{m-1}{2}}\pi_{\lambda,\omega}(f).
$$
Therefore, 
$$
\int_{\H^n}|\mathcal{L}_s^{1/2}(R_{\omega}(-\Delta_z)^{\frac{m-1}{4}}f)(v,t)|^2\,dv\,dt=C\int_{-\infty}^{\infty}\sum_{\alpha\in \Na^n}\frac{\Gamma\big(\frac{2|\alpha|+n+1+s}{2}\big)}{\Gamma\big(\frac{2|\alpha|+n+1-s}{2}\big)}\|\pi_{\lambda,\omega}(f)\|_2^2|\lambda|^{n+m-1+s}\,d\lambda.
$$
Integrating both sides over $\mathbb{S}^{m-1}$ we immediately see that
\begin{multline*}
\int_{\mathbb{S}^{m-1}}\|R_{\omega}(-\Delta_z)^{\frac{m-1}{4}}u\|^2_{H^{s+1}(\H^{n+1})}\,d\sigma(\omega)\\
=C\int_{\mathbb{S}^{m-1}}\int_{-\infty}^{\infty}\|\pi_{\lambda,\omega}(\mathcal{L}_s^{1/2}f)\|_{\operatorname{HS}}^2|\lambda|^{n+m-1}\,d\lambda\,d\sigma(\omega),
\end{multline*}
and from here we conclude.
\end{proof}
\subsection{Higher order extension problem}

In this subsection we will deal with the extension problem for large values of $s>0$. By Theorem \ref{thm:sol} we have that, for $s>0$, 
$$
u(v,z,\rho)=\Big(\frac{\pi^{n+m/2}}{4^m}\frac{\Gamma(s)\Gamma(\frac{n+s}{2})}{\Gamma(n+s)\Gamma\big(\frac{n+m+s}{2}\big)}\Big)^{-1}\rho^{2s} f\ast \varphi_{s,\rho}(v,z)
$$ 
solves the extension problem and $\lim_{\rho\to 0}u(v,z,\rho)= f(v,z)$. We are interested in recovering $\mathcal{L}_sf$ as the limit of certain derivatives (in $\rho$) of $u(v,z,\rho)$. Observe that the operator $\mathcal{L}_s$ is defined for all values of $s\in \R\setminus D$, where $D=\{\pm(n+2k+1): k=0,1,2,\ldots\}=D^+\cup D^-$.  Higher order extension problems have been studied in \cite{CG, Y} for the  Euclidean Laplacian and $s\in (0,n/2)$, and for general non-negative self-adjoint linear operators defined in an $L^2$-space and any noninteger positive number $s$, in \cite{RS}.

The main result of this subsection reads as follows.
\begin{thm}
\label{thm:higher}
Let $N$ be a $H$-type group with homogeneous dimension $Q=2(n+m)$. Let $s>0$ be such that either $s\notin \Na$ if $m$ is even, or $s\notin D^+$ if $m$ is odd. Suppose that $\ell$ is the integer such that $\ell-1\le  s<\ell$. Let $u$ be the solution of the extension problem \eqref{eq:ep} with initial condition $f\in L^p(N)$, $1\le p<\infty$. Then
$$
\lim_{\rho\to 0}\rho^{2(\ell-s)}\big(\frac{1}{2\rho}\partial_{\rho}\big)^{\ell}u(v,z,\rho)=C(\ell, m,n,s)\mathcal{L}_sf(v,z),
$$
in the $L^p$ norm provided $\mathcal{L}_sf\in L^p(N)$, where $C(\ell, m,n,s)=C_1(n,m,s)C_2(n,m,s)^{-1}a(n,m,s)$ with $C_1(n,m,s)$ and $C_2(n,m,s)$ given, respectively, by \eqref{eq:C1} and \eqref{eq:C2}, and $a(n,m,s)$ by 
$$
a(n,m,s)=\frac{\pi^{n+m/2}}{4^m}\frac{\Gamma(\ell-s)}{\Gamma(n+\ell-s)}\sum_{j=0}^{\ell}c(\ell,j)\frac{\Gamma(\frac{n+\ell+j-s}{2})}{\Gamma\big(\frac{n+m+\ell+j-s}{2}\big)}.
$$
Here $c(\ell, j)$ satisfies the recurrence relation
\begin{align}
\label{eq:recu}
\notag c(\ell+1,j)&=(j+1)c(\ell,j+1)-\frac12(n+m+\ell+j-1-s)c(\ell,j-1), \quad 1\le j\le \ell,\\
c(\ell+1,0)&=c(\ell,1),\qquad c(\ell+1,\ell+1)=-\frac12(n+m+2\ell-s)c(\ell, \ell).
\end{align}
\end{thm}
\begin{proof}
We can write the solution (see Remark \ref{rem:solution}) as 
$$
u(v,z,\rho)=C_1(n,m,s)C_2(n,m,s)^{-1}\mathcal{L}_sf\ast \varphi_{-s,\rho}(v,z),
$$
provided $0<s<n+1$.
By induction, still assuming $0<s<n+1$, we can show that
\begin{equation}
\label{eq:1h}
\big(\frac{1}{2\rho}\partial_{\rho}\big)^{\ell}\varphi_{-s,\rho}(v,z)=\sum_{j=0}^{\ell}c(\ell,j)g_{j,\rho}(v,z)\varphi_{\ell-s,\rho}(v,z),
\end{equation}
for some explicit constants $c(\ell,j)$ that will be computed at the end of the proof and 
$$
g_{j,\ell}(v,z)=\frac{\big(\rho^2+|v|^2\big)^{j}}{\big((\rho^2+|v|^2)^2+16|z|^2\big)^{j/2}}.
$$
This shows that
\begin{equation}
\label{eq:iden}
\rho^{2(\ell-s)}\big(\frac{1}{2\rho}\partial_{\rho}\big)^{\ell}u(v,z,\rho)=C_1(n,m,s)C_2(n,m,s)^{-1}\sum_{j=0}^\ell \mathcal{L}_sf\ast  h_{j,\rho,s}(v,z),
\end{equation}
where
$$
h_{j,\rho,s}(v,z)=c(\ell,j)g_{j,\rho}(v,z)\rho^{2(\ell-s)}\varphi_{\ell-s,\rho}(v,z).
$$
We claim that the latter identity is valid for any $0<s\notin \Na$ or $s\notin D^+$ as the case may be. This restriction comes from the fact that $C_2(n,m,s)$ is not defined when $s$ is of the form $n+m+2k$, for a nonnegative integer $k$. In order to prove the claim, we use Fourier transform on $N$. As $u(v,z,\rho)=C_1(n,m,s)\rho^{2s}f\ast \varphi_{s,\rho}(v,z),$ by Theorem \ref{thm:CH}, it is enough to show that
$$
\pi_{\lambda, \omega}(f) \rho^{2(\ell-s)}\big(\frac{1}{2\rho}\partial_{\rho}\big)^{\ell}[\rho^{2s}\pi_{\lambda, \omega}(\varphi_{{s,\rho}})]=C_2(n,m,s)^{-1}\pi_{\lambda, \omega}(\mathcal{L}_sf)\sum_{j=0}^{\ell}\pi_{\lambda, \omega}(h_{j,\rho,s})
$$
for any $\lambda>0$ and $\omega\in \mathbb{S}^{\ell-1}$.
Since $\varphi_{s,\rho}(v,z)$ is radial in both $v$ and $z$, the group Fourier transform of $ \varphi_{s,\rho} $ is given by 
$$ \pi_{\lambda,\omega}(\varphi_{s,\rho}) =    \sum_{k =0}^\infty  c_{k,\rho}^\lambda(s) P_k(\lambda),
$$ 
see Subsection \ref{subsec:reprH}. Here the coefficients have been computed in \cite[Section 3]{RT} and are given by the formula:
\begin{equation}
\label{eq:coeff}
c_{k,\rho}^\lambda(s)  =  \frac{ (2\pi)^{n+1}  |\lambda|^s}{\Gamma\big(\frac12(n+1+s)\big)^2}  L\Big(\frac{\rho^2 |\lambda|}{4 }, \frac{2k+n+1+s}{2},\frac{2k+n+1-s}{2}\Big),
\end{equation}
where the $L$-function is
\begin{equation}
\label{eq:Lfunction}
L(a,b,c) = \int_{0}^\infty e^{-a(2x+1)} x^{b-1}\big(1+x\big)^{-c} dx,
\end{equation}
which is valid for $ a, b \in \R^+ $ and $ c \in \R $. Using the facts that $\varphi_{s,\rho}$ and $h_{j,\rho,s}$ are radial and
$$
 \pi_{\lambda,\omega}(\mathcal{L}_sf)= \pi_{\lambda,\omega}(f)(2|\lambda|)^s\frac{\Gamma\big(\frac{H(\lambda)}{2|\lambda|}+\frac{1+s}{2}\big)}{\Gamma\big(\frac{H(\lambda)}{2|\lambda|}+\frac{1-s}{2}\big)},
$$
we only need to show that
\begin{equation}
\label{eq:fside}
\rho^{2(\ell-s)}\big(\frac{1}{2\rho}\partial_{\rho}\big)^{\ell}[\rho^{2s}c_{k,\rho}^{\lambda}(s)]=C_2(n,m,s)^{-1}(2|\lambda|)^s\frac{\Gamma\big(\frac{2k+n+1+s}{2}\big)}{\Gamma\big(\frac{2k+n+1-s}{2}\big)}\sum_{j=0}^{\ell}\widehat{h}_{j,\rho,s}(\lambda,k)
\end{equation}
for every $k\in \Na$. Here,
$$
\widehat{h}_{j,\rho,s}(\lambda,k)=c_n\frac{k!(n-1)!}{(k+n-1)!}\int_{\C^n}h_{j,\rho,s}^{\lambda\omega}(v)\varphi_k^{\lambda}(v)\,dv
$$
(observe that $h_{j,\rho,s}^{\lambda\omega}(v)$ is independent of $\omega$ as $h_{j,\rho,s}(v,z)$ is radial in $z$).

From the definition of the $L$-function it is clear that $\rho^{2(\ell-s)}\big(\frac{1}{2\rho}\partial_{\rho}\big)^{\ell}[\rho^{2s}c_{k,\rho}^{\lambda}(s)]$ is a holomorphic function of $s$ on $\Re s>0$. On the other hand, it is easy to see that the right hand side of \eqref{eq:fside} is also holomorphic on $\{0<\Re s<\ell\}\cap(\C\setminus \Na)$ or $\{0<\Re s<\ell\}\cap(\C\setminus D^+)$ as the case may be. As both sides agree on $0<\Re s<n+1$, we can conclude that they agree on $0<s<\ell$. This proves the claim. 

Thus, we have proved \eqref{eq:iden} for $0<s<\ell$. As $h_{j,\rho,s}(v,z)=\rho^{-2(n+m)}h_{j,1,s}(\rho^{-1}v,\rho^{-2}z)$ and $h_{j,1,s}\in L^1(N)$, we see that
$$
\lim_{\rho\to0}\rho^{2(\ell-s)}\big(\frac{1}{2\rho}\partial_{\rho}\big)^{\ell}u(v,z,\rho)=C_1(n,m,s)C_2(n,m,s)^{-1}a(n,m,s)\mathcal{L}_sf,
$$
where, by Lemma \ref{lem:I} with $\alpha=\ell+j-s$,
\begin{align*}
\notag a(n,m,s)&=\sum_{j=0}^{\ell}c(\ell,j)\int_N(1+|v|^2)^{j}\big((1+|v|^2)^2+16|z|^2\big)^{-\frac{n+m+\ell+j-s}{2}}\,dv\,dz\\
&=\frac{\pi^{n+m/2}}{4^m}\frac{\Gamma(\ell-s)}{\Gamma(n+\ell-s)}\sum_{j=0}^{\ell}c(\ell,j)\frac{\Gamma(\frac{n+\ell+j-s}{2})}{\Gamma\big(\frac{n+m+\ell+j-s}{2}\big)}.
\end{align*}
Finally, we can get a recurrence relation for the constants $c(\ell,j)$. By rewriting \eqref{eq:1h}, we have
\begin{equation}
\label{eq:rew}
\big(\frac{1}{2\rho}\partial_{\rho}\big)^{\ell}\varphi_{-s,\rho}(v,z)=\sum_{j=0}^{\ell}c(\ell,j)(\rho^2+|v|^2)^{j}\big((\rho^2+|v|^2)^2+|z|^2\big)^{-\frac{n+m+\ell+j-s}{2}}.
\end{equation}
Differentiating the above once more, we get
\begin{multline*}
\big(\frac{1}{2\rho}\partial_{\rho}\big)^{\ell+1}\varphi_{-s,\rho}(v,z)=\sum_{j=1}^{\ell}c(\ell,j)j(\rho^2+|v|^2)^{j-1}\big((\rho^2+|v|^2)^2+|z|^2\big)^{-\frac{n+m+\ell+1+j-1-s}{2}}\\
-\frac12\sum_{j=0}^{\ell}c(\ell,j)(n+m+\ell+j-s)(\rho^2+|v|^2)^{j+1}\big((\rho^2+|v|^2)^2+|z|^2\big)^{-\frac{n+m+\ell+1+j+1-s}{2}}.
\end{multline*}
Comparing this with \eqref{eq:rew} we obtain the recurrence relation in \eqref{eq:recu}.
This completes the proof.
\end{proof}

\section{Trace Hardy and Hardy inequalities for  the sublaplacian}
\label{sec:Hardy}

Our aim in this section is to prove various forms of trace Hardy and Hardy inequalities for  $ \mathcal{L} $ on a $H$-type group $ N.$ 
Let us recall that $\nabla u=(X_1u, \ldots, X_{2n}u,  \frac12Z_1u, \ldots, \frac12Z_mu,\partial_\rho u)$ and let
\begin{equation}
\label{eq:opera}
\mathbb{L}:= -\mathcal{L}+\partial_\rho^2 +\frac{1-2s}{\rho}\partial_\rho +\frac{1}{4}\rho^2 \Delta_z
\end{equation}
stand for the extension operator. We have the following general trace Hardy inequality. In order to state the result, we recall the Sobolev space presented in the Introduction.  Let $\widetilde{W}_0^{s,2}(S)$ be the completion of $C_0^{\infty}(N\times \R)$ with respect to the norm 
$$
\|u\|_{(s)}^2=\int_0^{\infty}\int_N|\nabla u(v,z,\rho)|^2\rho^{1-2s}\,dv\,dz\,d\rho.$$ As it was noted in the Introduction, it can be easily checked that this is indeed a norm.

\begin{thm}[General trace Hardy inequality]
\label{thm:ThfHardyi}
Let $0<s<1$ and let $ \varphi \in L^2(N) $ be a real valued function in the domain of $\mathcal{L}_s$ such that $ \varphi^{-1} \mathcal{L}_s\varphi $ is locally integrable. Then for any  real valued function $u(v,z, \rho) \in\widetilde{W}_0^{s,2}(S)$ we have the inequality
$$ \int_0^\infty \int_{N}\Big| \N u(v,z, \rho)\Big|^2 \rho^{1-2s}\,dv\,dz\, d\rho\geq  2^{1-2s}\frac{\Gamma(1-s)}{\Gamma(s)}\int_{N} u^2(v,z,0) \frac{\mathcal{L}_s\varphi(v,z)}{\varphi(v,z)} \,dv\,dz.$$

\end{thm}

\begin{proof}[Proof of Theorem \ref{thm:ThfHardyi}]
Let $Y_i$ be any of the vector fields $X_j$ or $\frac12Z_k$ on the $N$ group.  An easy limiting argument shows that it is enough to prove the inequality for functions $ u $  which are restrictions to $ AN $ of a $ C_0^\infty $ function on $ N \times \R.$ Let $u(v,z,\rho)$ be such a function and take  $w(v,z,\rho)$ to be the solution of the extension problem with initial condition $ \varphi.$  Now consider the integral
$$\int_{N}\big(Y_iu-\frac{u}{w}Y_iw\big)^2 \,dv\,dz= \int_{N}\Big((Y_i u)^2-2\frac{u}{w}Y_iu Y_iw+\big(\frac{u}{w}Y_iw\big)^2\Big)\,dv\,dz.$$
Integrating by parts and noting that there are no boundary terms we get
\begin{align*}
\int_{N}\frac{u}{w}Y_iu Y_iw\,dv\,dz&=-\int_{N} u Y_i\big(\frac{u}{w}Y_iw \big)\,dv\,dz\\
&=-\int_{N}\frac{u}{w}Y_iuY_iw\,dv\,dz-\int_{N}u^2Y_i\big(\frac{1}{w}Y_iw \big)\,dv\,dz.
\end{align*}
Since $\int_{N}u^2Y_i\big(\frac{1}{w}Y_iw \big)\,dv\,dz=-\int_{N}\frac{u^2}{w^2}\big(Y_iw\big)^2\,dv\,dz+\int_{N}\frac{u^2}{w}Y^2_iw\,dv\,dz$, the above gives
$$
\int_{N}\Big(\frac{u^2}{w^2}(Y_i w)^2-2\frac{u}{w}Y_iu Y_iw \Big)\,dv\,dz=\int_{N}\frac{u^2}{w}Y_i^2 w\,dv\,dz.
$$
On the other hand, a similar calculation with the $\rho$-derivative gives
\begin{multline*}
\int_0^{\infty}\big(\frac{u^2}{w^2}(\partial_{\rho} w)^2-2\frac{u}{w}\pa_{\rho} u\pa_{\rho} w\big)\rho^{1-2s}\,d\rho =\int_0^{\infty}\frac{u^2}{w}\pa_\rho\big( \rho^{1-2s}\pa_\rho w\big)\,d\rho\\+\frac{u(v,z, 0)^2}{w(v,z, 0)}\,\lim_{\rho \rightarrow 0}\big(\rho^{1-2s}\pa_{\rho} w\big)(v,z, \rho).
\end{multline*}
Let us now add them up and  take all integrations into account. By denoting $x=(v,z)$ and $dx=dv\,dz$, we get, in view of \eqref{eq:opera}, 
\begin{multline}
\label{eq:key}
\int_0^\infty \int_{N}\Big| \N u(x,\rho) -\frac{u(x,\rho)}{w(x,\rho)} \N w(x,\rho)\Big|^2 \rho^{1-2s}\,dx\, d\rho=\int_0^\infty \int_{N}\Big| \N u(x,\rho)\Big|^2 \rho^{1-2s}\,dx\, d\rho\\+\int_0^\infty \int_{N} \frac{u(x,\rho)^2}{w(x,\rho)}\,(\mathbb{L}w(x,\rho)) \rho^{1-2s}\,dx\, d\rho+\int_{N} \frac{u(x,0)^2}{w(x,0)}\lim_{\rho \rightarrow 0}\rho^{1-2s} \pa_\rho w(x, \rho)\,dx.
\end{multline} 

Since we have taken $ w=C_1(n,m,s)\rho^{2s} \varphi\ast \varphi_{s,\rho} $, with $\varphi_{s,\rho}$ as in \eqref{eq:varfisr} and $C_1(n,m,s)$ as in \eqref{eq:C1},  $w$ solves the equation $\mathbb{L}w=0$, with $w(x,0)= \varphi(x).$ Moreover, as $\rho\to 0$,
$$ 
-\rho^{1-2s} \partial_\rho (C_1(n,m,s)\rho^{2s} \varphi\ast \varphi_{s,\rho}) \rightarrow 2^{1-2s}\frac{\Gamma(1-s)}{\Gamma(s)}\mathcal{L}_s  \varphi. 
$$ 
Therefore, \eqref{eq:key} simplifies and we obtain the inequality
$$
\int_0^\infty\int_{N}\Big| \N u(x, \rho)\Big|^2 \rho^{1-2s}\,dx d\rho\geq 2^{1-2s}\frac{ \Gamma(1-s)}{\Gamma(s)}\int_{N} u^2(x,0)\frac{\mathcal{L}_s \varphi(x)}{ \varphi(x)}\,dx,
$$ 
as desired.
\end{proof}

\begin{rem} By taking $\varphi(v,z)=u(v,z,0)$, we obtain the following inequality
$$ \int_0^\infty \int_{N}| \N u(v,z, \rho)|^2 \rho^{1-2s}\,dv\,dz\, d\rho\geq  C \int_{N} \mathcal{L}_s\varphi(v,z) \varphi(v,z) \,dv\,dz.$$
This has been already proved in \cite{FGMT} by using results from scattering theory.
\end{rem}

Theorem \ref{thm:ThfHardyi} proves the main part of Theorem \ref{thm:GtHi}. In order to show that the inequality is sharp we claimed that equality is attained when $ u $ is a solution of the extension problem with initial condition $ \varphi$. If only we know that the solution $ u $ belongs to the space $ \widetilde{W}_0^{s,2}(S)$, this will be easily seen by checking that both sides of the inequality reduces to $ (\varphi, \mathcal{L}_s\varphi)$. Thus we need the following result, which is in part the counterpart of Proposition~\ref{prop:Ws2} for $H$-type groups.

\begin{thm} 
\label{thm:Ws2N}
For $ 0 < s < 1, $ let $ \varphi $ be a real valued function in $  H^s(N), $ i.e. $ \varphi$, $\mathcal{L}_{s/2} \varphi \in L^2(N)$.  If $ u $ is the solution of the extension problem with initial condition $ \varphi,$ then 
  $ u \in \widetilde{W}_0^{s,2}(S)$. 
\end{thm}

In order to prove Theorem \ref{thm:Ws2N}, we need the following proposition: 

\begin{prop} 
\label{prop:equ}
Under the same hypothesis as in Theorem \ref{thm:Ws2N} we have
$$   \int_0^\infty \int_{N}| \N u(v,z, \rho)|^2 \rho^{1-2s}\,dv\,dz\, d\rho =2^{1-2s}\frac{\Gamma(1-s)}{\Gamma(s)} \int_{N} \mathcal{L}_s\varphi(v,z)\varphi(v,z) \,dv\,dz.
$$

\end{prop}
\begin{proof} 
The proof follows the lines of the proof of Proposition \ref{prop:1}: we take into account that $ Y_j \varphi_{s,\rho} \in L^2(N) $ when $ Y_j $ is any of the vector fields $ X_j $ or $ Z_j$ (and therefore $ Y_j u $ also vanishes at infinity), and the same for $ \partial_\rho \varphi_{s,\rho}$. Integrating by parts, we have that
$$  
\int_N |Y_ju(v,z,\rho)|^2 dv dz = \int_N u(v,z,\rho) Y_j^2u(v,z,\rho) dv dz.
$$
Moreover, $ |u(v,z,\rho)| \leq \rho^{2s} \|\varphi\|_2 \| \varphi_{s,\rho}\|_2 $, which implies that $ u(v,z,\rho) $ goes to $ 0 $ as $ \rho $ tends to infinity (the same for $ \partial_\rho u(v,z,\rho)$).
The computation with the $\rho$-derivative yields
\begin{multline*}
\int_0^{\infty}(\partial_{\rho} u(v,z,\rho))^2 \rho^{1-2s} d\rho   =\int_0^{\infty} u(v,z,\rho) \pa_\rho\big( \rho^{1-2s}\pa_\rho u(v,z,\rho)\big)\,d\rho\\- u(v,z, 0)\,\lim_{\rho \rightarrow 0}\big(\rho^{1-2s}\pa_{\rho} u\big)(v,z, \rho).
\end{multline*}
Now we sum up and we use the fact that $ u $ solves the extension problem with initial condition $ \varphi$. The result follows.
\end{proof}

From Proposition \ref{prop:equ} we see that  the ``energy norm'' of the solution $ u $ is a constant multiple of the $ H^s(N) $ norm of the initial condition. 

\begin{proof}[Proof of Theorem \ref{thm:Ws2N}] We are now in a position to prove Theorem \ref{thm:Ws2N}. As the energy norm of $ u $ is finite, all we have to do is to show that it can be approximated by a sequence of compactly supported smooth functions on $ N \times \R.$ As $ u(v,z,\rho) $ is even in $ \rho $ we can think of it as a smooth function on $ N \times \R.$ Let $ \eta \in C_0^\infty(\R) $ be supported in $  |\rho| \leq 2 $ and assume that $ \eta = 1 $ on $ |\rho| \leq 1.$ Let $ \psi_j(v,z,\rho) = \eta(2^{-j} ((\rho^2+|v|^2)^2+|z|^2))$ and define $ u_j(v,z,\rho) = \psi_j(v,z,\rho)u(v,z,\rho).$ We will show that $ u_j $ converges to $ u $ in the energy norm. Observe that $ \nabla(u-u_j) = (1-\psi_j) \nabla u+ u \nabla \psi_j .$ Since $ (1-\psi_j) $ is supported in $ ((\rho^2+|v|^2)^2+|z|^2) \geq 2^j $ it follows that 
\begin{multline*} 
\int_0^\infty \int_{N} |(1-\psi_j)(v,z,\rho)|^2 | \N u(v,z, \rho)|^2 \rho^{1-2s}\,dv\,dz\, d\rho\\
\leq C  \int_{ ((\rho^2+|v|^2)^2+|z|^2) \geq 2^j } | \N u(v,z, \rho)|^2 \rho^{1-2s}\,dv\,dz\, d\rho
\end{multline*}
which tends to $ 0 $ as $ j $ tends to infinity, in view of Proposition \ref{prop:equ}. On the other hand, $ \N \psi_j $ is supported on $ 2^j \leq  ((\rho^2+|v|^2)^2+|z|^2) \leq 2^{j+1} $. Moreover, we have that 
\begin{align*}
|\partial_{\rho}\psi_j(v,z,\rho)|&=2^{-j}2(\rho^2+|v|^2)2|\rho|\eta'(2^{-j}((\rho^2+|v|^2)^2+|z|^2))\\
&\le C 2^{-j}((\rho^2+|v|^2)^2+|z|^2)^{3/4}\eta'(2^{-j}((\rho^2+|v|^2)^2+|z|^2))\\
&=C 2^{-j3/4}((\rho^2+|v|^2)^2+|z|^2)^{3/4}\eta'(2^{-j}((\rho^2+|v|^2)^2+|z|^2))2^{-j/4}\\
&=C [2^{-j}((\rho^2+|v|^2)^2+|z|^2)]^{3/4}\eta'(2^{-j}((\rho^2+|v|^2)^2+|z|^2))2^{-j/4}.
\end{align*}
Therefore, calling $t:=2^{-j}((\rho^2+|v|^2)^2+|z|^2))$, we have 
$$
|\partial_{\rho}\psi_j(v,z,\rho)|\le C t^{3/4}\eta'(t)2^{-j/4}\le C 2^{-j/4}.
$$
Same estimate is true of $ X \psi_j(v,z,\rho) $ for any other vector field $ X.$  For instance, let us compute and estimate $|X_k \psi_j(v,t,\rho)|$ on the Heisenberg group  with $v=x+iy\in \C^n$ and $t\in \R$ for simplicity. We get
\begin{align*}
|X_k \psi_j(v,t,\rho)|&\le 2^{-j}\big(4(\rho^2+|v|^2)|x_k|+|y_k||t|\big)\eta'(2^{-j}((\rho^2+|v|^2)^2+|z|^2))\\
&\le C2^{-j}((\rho^2+|v|^2)^2+t^2)^{1/2}((\rho^2+|v|^2)^2+t^2)^{1/4}\eta'(2^{-j}((\rho^2+|v|^2)^2+|z|^2))\\
&=C [2^{-j}((\rho^2+|v|^2)^2+t^2)]^{3/4}\eta'(2^{-j}((\rho^2+|v|^2)^2+|z|^2))2^{-j/4}.
\end{align*}
As before, we infer that
\begin{equation}
\label{eq:dersmall}
| \N \psi_j(v,z,\rho)| \leq C 2^{-j/4}. 
\end{equation}
Since the solution $u$   of the extension problem with initial condition $\varphi$  is given by $u= C_1(n,m,s)\rho^{2s}\varphi\ast \varphi_{s,\rho},$ by Young's inequality, we get $\|u\|_{L^2(N)}\le \|\varphi\|_{L^2(N)}$ by the choice of the constant $ C_1(n,m,s).$
Now  as $ \N \psi_j $ is supported in $ 2^j \leq  ((\rho^2+|v|^2)^2+|z|^2) \leq 2^{j+1} $
\begin{multline*}
\int_0^\infty \int_{N} |u(v,z,\rho)|^2 | \N \psi_j(v,z, \rho)|^2 \rho^{1-2s}\,dv\,dz\, d\rho\\
 \leq C 2^{-j/2} \int_0^{2^{(j+1)/4}} \Big( \int_{N} |u(v,z,\rho)|^2  dv\,dz\,\Big) \rho^{1-2s}\, d\rho \leq C 2^{-j/2}\int_0^{2^{(j+1)/4}}\rho^{1-2s}\, d\rho
\end{multline*}
which clearly goes to zero as $ j $ tends to infinity since $ s > 0.$
This completes the proof of Theorem \ref{thm:Ws2N}.

\end{proof}

We are going to show another form of trace Hardy inequality. We make use of the connection between solutions of the extension problem and eigenfunctions of the Laplace-Beltrami operator on $ S $ which we have already exploited.
Recall that  $$ \Delta_S =  \sum_{j=1}^{2n} E_j^2 +\sum_{k=1}^m T_k^2 +H^2 -\frac{1}{2}Q H.
$$
Given a function $u$ on $S,$  as before, we define  $ w(v,z,\rho) = u(2^{-1/2}v,2^{-1}z,\sqrt{2\rho})$ and $ \widetilde{w}(v,z,\rho) = \rho^{\frac{(n+m-s)}{2}}w(v,z,\rho).$ We also denote by $ \N_Sw $ the full gradient of $ w $ on $ S.$ Making use of the connection between the two gradients, we see that
$$ \int_0^\infty \int_{N}| \N u(v,z, \rho)|^2 \rho^{1-2s}\,dv\,dz\, d\rho = C \int_0^{\infty}\int_{N}|\nabla_S w(v,z,\rho)|^2\rho^{-1-s}\,dv\,dz\,d\rho $$
 whenever $ u $ and $ w $ are related as above. We now prove the following proposition.

 \begin{prop}
\label{thm:tH2}
Let $0<s<1$. Assume that $ w $ is the restriction to $ S $ of a function in $ C_0^\infty(N\times \R).$ Then we have the identity
\begin{align*}
&\int_0^{\infty}\int_{N}|\nabla_S w(v,z,\rho)|^2\rho^{-1-s}\,dv\,dz\,d\rho\\
&\quad =\int_0^{\infty}\int_N\widetilde{w}(v,z,\rho)\big(-\Delta_S-\frac{(n+m)^2-s^2}{4}\big)\widetilde{w}(v,z,\rho)\rho^{-n-m-1}\,dv\,dz\,d\rho\\
&\qquad -\int_{N} w(v,z,0)\lim_{\rho \rightarrow 0}\rho^{1-s} \pa_\rho w(v,z, \rho)\,dv\,dz\,d\rho.
\end{align*} 
 
\end{prop}
\begin{proof} First of all, note that the assumptions on $ s $ and $ w $ and the definition of $ \nabla_S $ ensures that all the integrals involved in Proposition \ref{thm:tH2} are finite.
By integration by parts we have, as the boundary terms vanish,
\begin{equation}
\label{eq:1b}
\int_N\sum_{j=1}^{2n}E_j^2\widetilde{w}(v,z,\rho) \widetilde{w}(v,z,\rho)\rho^{-n-m-1}\,dv\,dz\,d\rho=-\int_N\sum_{j=1}^{2n}|E_jw|^2\rho^{-1-s}\,dv\,dz\,d\rho
\end{equation}
and also
\begin{equation}
\label{eq:2b}
\int_N\sum_{k=1}^{m}T_k^2\widetilde{w}(v,z,\rho) \widetilde{w}(v,z,\rho)\rho^{-n-m-1}\,dv\,dz\,d\rho=-\int_N\sum_{k=1}^{m}|T_kw|^2\rho^{-1-s}\,dv\,dz\,d\rho.
\end{equation}
Now, observe that
$$
H^2-(n+m)H=\rho^2\partial_{\rho}^2-(n+m-1)\rho\partial_{\rho}.
$$
Therefore, we consider the integral
\begin{multline}
\label{eq:1u}
\int_0^{\infty}(\rho^2\partial_{\rho}^2-(n+m-1)\rho\partial_{\rho})\widetilde{w}(v,z,\rho) \widetilde{w}(v,z,\rho)\rho^{-m-n-1}d\rho\\=\int_0^{\infty}\partial_{\rho}^2(\rho^{\frac{n+m-s}{2}}w)\rho^{-\frac{n+m+s-2}{2}}w\,d\rho-(n+m-1)\int_0^{\infty}\partial_{\rho}(\rho^{\frac{n+m-s}{2}}w)\rho^{-\frac{n+m+s}{2}}w\,d\rho.
\end{multline}
We first look at the truncated  integral 
$$
I:=\int_A^{B}\partial_{\rho}^2(\rho^{\frac{n+m-s}{2}}w)\rho^{-\frac{n+m+s-2}{2}}w\,d\rho
$$
where $0 <  A<B\le \infty$. Integration by parts and some computations show that
\begin{multline}
\label{eq:2u}
I=-\frac14 \big(2s(1-s)-(n+m-s)(n+m+s-2)\big)\int_A^{B}\rho^{-(1+s)}w^2d\rho\\-\int_A^{B}\rho^{1-s}(\partial_{\rho}w)^2\,d\rho+\rho^{-\frac{n+m+s-2}{2}}w\partial_{\rho}(\rho^{\frac{n+m-s}{2}}w)\Big|^B_A-\frac{1-s}{2}\rho^{-2}w^2\Big|^B_A.
\end{multline}
Concerning the second integral, we have
$$
II:=-(n+m-1)\int_0^{\infty}\partial_{\rho}(\rho^{\frac{n+m-s}{2}}w)\rho^{-\frac{n+m+s}{2}}w\,d\rho,
$$
which by integration by parts and some calculations leads to
\begin{equation}
\label{eq:3u}
II=-(n+m-1)\big(\frac{n+m}{2}\big)\int_A^{B}\rho^{-(1+s)}w^2d\rho-\frac{n+m-1}{2}\rho^{-2}w^2\Big|^B_A.
\end{equation}
Then, collecting \eqref{eq:1u}, \eqref{eq:2u} and \eqref{eq:3u}, we have
\begin{align*}
&\int_A^{B}(\rho^2\partial_{\rho}^2-(n+m-1)\rho\partial_{\rho})\widetilde{w}(v,z,\rho) \widetilde{w}(v,z,\rho)\rho^{-m-n-1}d\rho\\
&\quad =-\frac14\big((n+m)^2-s^2\big)\int_A^B\rho^{-(1+s)}w^2\,d\rho-\int_A^B\rho^{1-s}(\partial_{\rho}w)^2\,d\rho\\
&\qquad+\rho^{-\frac{n+m+s-2}{2}}w\partial_{\rho}[\rho^{\frac{n+m-s}{2}}w]\Big|^B_A-\frac{n+m-s}{2}\rho^{-2}w^2\Big|^B_A.
\end{align*}
Finally, observe that 
$$
\rho^{-\frac{n+m+s-2}{2}}w\partial_{\rho}(\rho^{\frac{n+m-s}{2}}w)\Big|^B_A=\frac{n+m-s}{2}\rho^{-2}w^2\Big|^B_A+w\rho^{1-s}\partial_{\rho}w\Big|^B_A,
$$
and consequently we obtain
\begin{align}
\label{eq:3b}
\notag&\int_0^{\infty}(H^2-(n+m)H)\widetilde{w}(v,z,\rho) \widetilde{w}(v,z,\rho)\rho^{-m-n-1}d\rho\\
&=\int_0^{\infty}(\rho^2\partial_{\rho}^2-(n+m-1)\rho\partial_{\rho})\widetilde{w}(v,z,\rho) \widetilde{w}(v,z,\rho)\rho^{-m-n-1}d\rho\\
\notag&=-\frac14\big((n+m)^2-s^2\big)\int_0^\infty\rho^{-(1+s)}w(v,z,\rho)^2\,d\rho-\int_0^\infty\rho^{-1-s}(\rho\partial_{\rho}w(v,z,\rho))^2\,d\rho\\
\notag&\qquad+w(v,z,0)\lim_{\rho \rightarrow 0}\rho^{1-s} \pa_\rho w(v,z, \rho).
\end{align}
Taking into account of \eqref{eq:1b}, \eqref{eq:2b} and \eqref{eq:3b}, we obtain the conclusion. 
\end{proof}
From Proposition \ref{thm:tH2} we can immediately obtain the following trace inequality.

\begin{thm}
\label{thm:tH2m}
Let $0<s<1$. We have the inequality
$$
\int_0^{\infty}\int_{N}|\nabla_S w(v,z,\rho)|^2\rho^{-1-s}\,dv\,dz\,d\rho\ge -\int_{N} w(v,z,0)\lim_{\rho \rightarrow 0}\rho^{1-s} \pa_\rho w(v,z, \rho)\,dv\,dz\,d\rho,
$$
 for any real valued function $w(v,z,\rho)=u(2^{-1/2}v,2^{-1}z,\sqrt{2\rho})$ with $ u \in \widetilde{W}_0^{s,2}(S).$  Moreover, equality holds if and only if $u$ is a solution of the extension problem \eqref{eq:ep}. 
\end{thm}

\begin{proof} The stated inequality for $ w \in C_0^\infty(S) $ follows from the above proposition since $ -\Delta_S \geq \frac{1}{4} (n+m)^2.$  By approximating $ u $ and hence $ w$ by a sequence of $ C_0^\infty(S) $ functions, we can conclude that the inequality remains true under the hypothesis on $ w$.
We have already remarked in Subsection \ref{sub:cha} that when $u$ satisfies the extension problem, then $ \widetilde{w}(v,z,\rho) = \rho^{\frac{(n+m-s)}{2}}w(v,z,\rho) $ is an eigenfunction of $\Delta_S$ with eigenvalue $-\frac{(n+m)^2-s^2}{4}$. Consequently, we can easily conclude that equality holds if and only if $ u $ is the solution of the extension problem.
\end{proof}

Theorems \ref{thm:ThfHardyi} and \ref{thm:Ws2N} lead to some interesting corollaries.

\begin{cor}
\label{thm:traceH}
Let $0<s<1$, $\delta>0$. Then for all real valued $ u \in \widetilde{W}_0^{s,2}(S),$ we have 
\begin{equation}
\label{eq:tHn}
\int_0^{\infty}\int_{N}|\nabla u(v,z,\rho)|^2\rho^{1-2s}\,dv\,dz\,d\rho\ge C_{n,s}\delta^{2s}\int_{N}\frac{u(v,z,0)^2}{\big((\delta^2+|v|^2)^2+16z^2\big)^s}\,dv\,dz
\end{equation}
with the constant given by
$$
C_{n,s}=\frac{2^{1-2s}\Gamma(1-s)}{\Gamma(s)}4^{2s}\frac{\Gamma(\frac{n+1+s}{2})}{\Gamma(\frac{n+1-s}{2})}\frac{\Gamma(\frac{n+m+s}{2})}{\Gamma(\frac{n+m-s}{2})}.
$$
The above inequality is sharp and equality is obtained when $u(v,z,\rho)=\varphi_{-s,\delta}\ast \rho^{2s}\varphi_{s,\rho}(v,z)$.
\end{cor}

\begin{proof}
We immediately obtain the inequality by taking $\varphi=\varphi_{-s,\delta}$ in Theorem \ref{thm:ThfHardyi} for a fixed $\delta>0$ and using Theorem~\ref{thm:CH}.  

As for the equality, when we take $u(v,z,\rho)=C_1(n,m,s)\varphi_{-s,\delta}\ast \rho^{2s}\varphi_{s,\rho}(v,z)$ in the inequality \eqref{eq:tHn}, in view of Theorem \ref{thm:Ws2N}, the left hand side reduces to $2^{1-2s}\frac{ \Gamma(1-s)}{\Gamma(s)}\varphi_{-s,\delta}\mathcal{L}_s \varphi_{-s,\delta} $ which, by Theorem~\ref{thm:CH}, is nothing but the right hand side of \eqref{eq:tHn}.
\end{proof}

\begin{rem}
In Corollary \ref{thm:traceH}, if we take $ u $ to be the solution of the extension problem with initial condition $ f $, the left hand side reduces to a constant multiple of $ (f,\mathcal{L}_s f).$ This immediately proves Corollary \ref{cor:HarNH}.
\end{rem}

Now we are going to prove  Theorem \ref{thm:tHihomo}. Let us recall some definitions.
Let
$$
\varphi_s(v,z)=|(v,z)|^{-(n+m+s)}\quad \text{ and } \quad \psi_s(v,z)=C_1(n,m,s)(\varphi_s\ast |\cdot|^{-Q+2s})(v,z). 
$$
Note that $\psi_s(v,z)$ is homogeneous of degree $-(n+m-s)$. Let us define 
\begin{equation*}
w_s(v,z)=\varphi_s(v,z)\psi_s(v,z)^{-1}
\end{equation*} 
so that $w_s$ is homogeneous of degree $-2s$.

\begin{proof}[Proof of Theorem \ref{thm:tHihomo}]
Let $u\in C_0^{\infty}(N\times \R)$ and let $\varphi>0$ be any function which is in the domain of $\mathcal{L}_s$. Let $w=C_1(n,m,s)\rho^{2s}\varphi\ast \varphi_{s,\rho}$ be the solution of the extension problem \eqref{eq:ep} with initial condition $\varphi$. Then, proceeding as in the proof of Theorem \ref{thm:ThfHardyi}, we get
\begin{multline}
\label{eq:key2}
\int_0^\infty \int_{N}\Big| \frac{\N u(x,\rho)}{u(x,\rho)}-\frac{\N w(x,\rho)}{w(x,\rho)} \Big|^2 \rho^{1-2s}u^2(x,\rho)\,dx\, d\rho\\+\frac{2^{1-2s}\Gamma(1-s)}{\Gamma(s)}\int_{N} \frac{\mathcal{L}_s\varphi(x)}{\varphi(x)}u^2(x,0)\,dx
=\int_0^\infty \int_{N}\Big| \frac{\N u(x,\rho)}{u(x,\rho)}\Big|^2 \rho^{1-2s}u^2(x,\rho)\,dx\, d\rho.
\end{multline} 
Now, let $\eta\in C_0^{\infty}$ supported in $0< a \le |(v,z)|\le b<\infty$ be such that $0\le \eta \le 1$. Let us take $\varphi=\eta\varphi_s\ast \varphi_{-s,\delta}$ in \eqref{eq:key2}. We note that, in view of Theorem \ref{thm:CH}, 
\begin{multline*}
\frac{2^{1-2s}\Gamma(1-s)}{\Gamma(s)}\int_{N} \frac{\mathcal{L}_s\varphi(x)}{\varphi(x)}u^2(x,0)\,dx\\
=\frac{2^{1-2s}\Gamma(1-s)}{\Gamma(s)}C_2(n,m,s)\int_N\frac{\delta^{2s}(\eta\varphi_s)\ast \varphi_{s,\delta}(v,z)}{\eta\varphi_s\ast \varphi_{-s,\delta}(v,z)}u^2(v,z,0)\,dv\,dz.
\end{multline*}

From identity \eqref{eq:key2} we observe that
\begin{multline*}
\frac{2^{1-2s}\Gamma(1-s)}{\Gamma(s)}C_2(n,m,s)\int_N\frac{\delta^{2s}(\eta\varphi_s)\ast \varphi_{s,\delta}(v,z)}{\eta\varphi_s\ast \varphi_{-s,\delta}(v,z)}u^2(v,z,0)\,dv\,dz\\
\le \int_0^\infty \int_{N}\Big| \N u(x,\rho)\Big|^2 \rho^{1-2s}\,dx\, d\rho.
\end{multline*}
As $\eta\varphi_s\in L^2(N)$ and $C_1(n,m,s)\delta^{2s}\varphi_{s,\delta}$ is an approximate identity, $\delta^{2s}\eta\varphi_s\ast \varphi_{s,\delta}$ converges to $C_1(n,m,s)^{-1}\eta \varphi_s(v,z)$ a.e. as $\delta\to 0$. We also note that $\eta\varphi_s\ast \varphi_{-s,\delta}(v,z)$ converges to $\eta\varphi_s\ast |\cdot|^{-Q+2s}$ (due to Dominated Convergence Theorem). By Vitali's theorem (see \cite[p. 143]{Ru}) with $d\mu(v,z)=u^2(v,z,0)\,dv\,dz$ we conclude that 
\begin{multline*}
\lim_{\delta\to0}\frac{2^{1-2s}\Gamma(1-s)}{\Gamma(s)}C_2(n,m,s)\int_N\frac{\delta^{2s}(\eta\varphi_s)\ast \varphi_{s,\delta}(v,z)}{\eta\varphi_s\ast \varphi_{-s,\delta}(v,z)}u^2(v,z,0)\,dv\,dz\\
=\frac{2^{1-2s}\Gamma(1-s)}{\Gamma(s)}C_2(n,m,s)\int_N\frac{\eta\varphi_s(v,z)}{C_1(n,m,s)\eta\varphi_s\ast |\cdot|^{-Q+2s}(v,z)}u^2(v,z,0)\,dv\,dz.
\end{multline*}
We also have, with $w_{\eta,\delta}(v,z,\rho)=C_1(n,m,s)\rho^{2s}\eta\varphi_s\ast \varphi_{-s,\delta}\ast \varphi_{s,\rho}(v,z)$,
\begin{multline*}
\int_0^\infty \int_{N}\Big| \frac{\N u(x,\rho)}{u(x,\rho)}-\frac{\N w_{\eta,\delta}(x,\rho)}{w_{\eta,\delta}(x,\rho)} \Big|^2 \rho^{1-2s}u^2(x,\rho)\,dx\, d\rho\\
\le \int_0^\infty \int_{N}\Big| \frac{\N u(x,\rho)}{u(x,\rho)}\Big|^2 \rho^{1-2s}u^2(x,\rho)\,dx\, d\rho.
\end{multline*}
Note that, as $\delta\to 0$, $w_{\eta,\delta}$ converges to 
$$
w_{\eta}(v,z,\rho)=C_1(n,m,s)\rho^{2s}\eta \varphi_s\ast |\cdot|^{-Q+2s}\ast \varphi_{s,\rho}(v,z)
$$ 
and $\N w_{\eta,\delta}\to \N w_{\eta}$. Again, by Vitali's theorem with $d\mu(v,z,\rho)=u^2(v,z,\rho)\rho^{1-2s}\,dv\,dz\,d\rho$, we can conclude that
\begin{multline*}
\lim_{\delta\to0}\int_0^\infty \int_{N}\Big| \frac{\N u(x,\rho)}{u(x,\rho)}-\frac{\N w_{\eta,\delta}(x,\rho)}{w_{\eta,\delta}(x,\rho)} \Big|^2 \rho^{1-2s}u^2(x,\rho)\,dx\, d\rho\\
=\int_0^\infty \int_{N}\Big| \frac{\N u(x,\rho)}{u(x,\rho)}-\frac{\N w_{\eta}(x,\rho)}{w_{\eta}(x,\rho)} \Big|^2 \rho^{1-2s}u^2(x,\rho)\,dx\, d\rho. 
\end{multline*}
Putting together, we have obtained
\begin{multline*}
\int_0^\infty \int_{N}\Big| \N u(x,\rho)\Big|^2 \rho^{1-2s}\,dx\, d\rho
=\int_0^\infty \int_{N}\Big| \frac{\N u(x,\rho)}{u(x,\rho)}-\frac{\N w_{\eta}(x,\rho)}{w_{\eta}(x,\rho)} \Big|^2 \rho^{1-2s}u^2(x,\rho)\,dx\, d\rho\\
+\frac{2^{1-2s}\Gamma(1-s)}{\Gamma(s)}C_2(n,m,s)\int_N\frac{\eta\varphi_s(v,z)}{C_1(n,m,s)\eta\varphi_s\ast |\cdot|^{-Q+2s}(v,z)}u^2(v,z,0)\,dv\,dz.
\end{multline*}
Choosing a sequence $\eta=\eta_k$ of functions supported on $\frac{1}{2k}\le |(v,z)|\le 2k$ which are equal to $1$ on $\frac{1}{k}\le |(v,z)|\le k$, arguing as above, we can take limit as $k\to \infty$ in the above to get
\begin{multline}
\label{eq:ideC}
\int_0^\infty \int_{N}\Big| \N u(x,\rho)\Big|^2 \rho^{1-2s}\,dx\, d\rho
=\int_0^\infty \int_{N}\Big| \frac{\N u(x,\rho)}{u(x,\rho)}-\frac{\N w(x,\rho)}{w(x,\rho)} \Big|^2 \rho^{1-2s}u^2(x,\rho)\,dx\, d\rho\\
+\frac{2^{1-2s}\Gamma(1-s)}{\Gamma(s)}C_2(n,m,s)\int_N u^2(v,z,0)w_s(v,z)\,dv\,dz.
\end{multline}

Let us take now $u\in \widetilde{W}_0^{s,2}(S)$. Choose a sequence $u_k\in C_0^{\infty}(N\times \R)$ such that $u_k$ converges to $u$ in $\widetilde{W}_0^{s,2}(S)$. It is clear that, passing to the limit in \eqref{eq:ideC} we get the identity for functions in $\widetilde{W}_0^{s,2}(S)$. From \eqref{eq:ideC} we deduce immediately the inequality stated in the theorem.

The equality is obtained if and only if $\big| \frac{\N u(x,\rho)}{u(x,\rho)}-\frac{\N w(x,\rho)}{w(x,\rho)} \big|=0$, i.e., if and only if $u=c\cdot w$, for some positive constant $c$, with $w=C_1(n,m,s)\rho^{2s}\varphi_s\ast |\cdot|^{-Q+2s}\ast \varphi_{s,\rho}(v,z)$. But this $w$ makes the rest of the terms of \eqref{eq:ideC} infinite. 
\end{proof}

\begin{proof}[Proof of Corollary \ref{thm:Hardyhom}]

Take $u(v,z,\rho)=C_1(n,m,s)\rho^{2s}f\ast \varphi_{s,\rho}(v,z,\rho)$. By Theorem \ref{thm:sol}, $u$ solves the extension problem \eqref{eq:ep} with initial condition $\varphi$ so that, by Theorem \ref{thm:GtHi}, we have the identity
$$
\int_0^\infty\int_{N}\big| \N u(x, \rho)\big|^2 \rho^{1-2s}\,dx d\rho=2^{1-2s}\frac{ \Gamma(1-s)}{\Gamma(s)}\int_Nf(x) \mathcal{L}_sf(x)\,dx.
$$
Then, applying Theorem \ref{thm:tHihomo}, we get 
$$
\int_Nf(x) \mathcal{L}_s(x)\,dx\ge  C_2(n,m,s)\int_N f^2(v,z)w_s(v,z)\,dv\,dz,
$$
and the conclusion follows.
\end{proof}

\begin{rem}
It would be interesting to get an inequality of the form
$$
\int_0^\infty\int_{N}\big| \N u(x, \rho)\big|^2 \rho^{1-2s}\,dx d\rho\ge C(n,m,s)\int_Nu^2(x,0)|x|^{-2s}\,dx,
$$
with sharp constant. Unfortunately, the weight function $w_s(x)$, though it has the right homogeneity, does not simplify to yield $|x|^{-2s}$, even in the case of the Heisenberg group. Recall that in the Euclidean setting 
$$
w_s(x)=\frac{|x|^{-\frac{n+s}{2}}}{|x|^{-\frac{n+s}{2}}\ast |x|^{-n+s}}=C(n,s)|x|^{-2s},
$$
with a precise constant $C(n,s)$ yielding the sharp Hardy inequality.

In the case of the Heisenberg group, we have 
$$
\psi_s(z,t)=\int_{\H^n}|(\zeta,\tau)|^{-(n+1+s)}|(z,t)(\zeta,\tau)^{-1}|^{-Q+2s}\,d\zeta\,d\tau=|(z,t)|^{-(n+1-s)}\psi_s(z',t'),
$$  
where $ |z'|^4+t'^2 = 1.$  Using Cayley transform (see for instance \cite{STU}), it is easy to see that the above integral is equal to a constant times
$$  ((1+|z'|^2)^2+t'^2)^{-\lambda / 4}  \int_{S^{2n+1}} |1-\zeta \cdot \bar{\eta}|^{-\lambda /2} |(1-\zeta_{n+1})(1+\zeta_{n+1})|^{-\gamma /2} d\zeta $$  where
$ \lambda = 2(n+1)-2s, \gamma = n+1+s$  and $ \eta $ is the Cayley transform of $ (z',t').$ A sharp  upper bound for the above integral would lead to a Hardy inequality with homogeneous weight $ |(z,t)|^{-2s}$.
\end{rem}

\begin{rem}
In Theorem \ref{thm:tHihomo}, the weight function $w_s(x)$ is optimal in the following sense: If $\widetilde{w}_s\ge w_s$ is another weight function for which
$$
\int_0^{\infty}\int_N \big| \N u(x, \rho)\big|^2 \rho^{1-2s}\,dx d\rho\ge 2^{1-2s}\frac{\Gamma(1-s)}{\Gamma(s)}C_2(n,m,s)\int_Nu^2(x,0)\widetilde{w}_s(x)\,dx
$$
then $\widetilde{w}_s=w_s$.

To see this, consider
\begin{multline*}
2^{1-2s}\frac{\Gamma(1-s)}{\Gamma(s)}C_2(n,m,s)\int_N(\widetilde{w}_s(x)-w_s(x))u^2(x,0)\,dx\\
\le \int_0^{\infty}\int_N \big| \N u(x, \rho)\big|^2 \rho^{1-2s}\,dx d\rho-2^{1-2s}\frac{\Gamma(1-s)}{\Gamma(s)}C_2(n,m,s)\int_Nw_s(x)u^2(x,0)\,dx.
\end{multline*}
In view of the identity \eqref{eq:ideC}, we have
\begin{equation}
\label{eq:ideo}
\int_N(\widetilde{w}_s(x)-w_s(x))u^2(x,0)\,dx \le \int_0^\infty \int_{N}\Big| \frac{\N u(x,\rho)}{u(x,\rho)}-\frac{\N w(x,\rho)}{w(x,\rho)} \Big|^2 \rho^{1-2s}u^2(x,\rho)\,dx\, d\rho
\end{equation}
where 
$$
w(x,\rho)=C_1(n,m,s)\rho^{2s}\varphi_s\ast |\cdot|^{-Q+2s}\ast \varphi_{s,\rho}(x).
$$ 
Let  $\psi_s=C_1(n,m,s)\varphi_s\ast |\cdot|^{-Q+2s}$, as in the proof of Theorem \ref{thm:tHihomo}. It is clear that $\psi_s(x)\le C|x|^{-\frac{Q}{2}+s}$. By defining $\psi_s^{(1)}=\psi_s\chi_{|x|\le 1}$ and $\psi_s^{(2)}=\psi_s\chi_{|x|> 1}$, we see that $\psi_s=\psi_s^{(1)}+\psi_s^{(2)}\in L^1(N)+L^p(N)$ for $p>\frac{2Q}{Q-2s}$. Consequently, $C_1(n,m,s)\rho^{2s}\psi_s\ast \varphi_{s,\rho}(x)\to \psi_s(x)$ for a.e. $x$ and also in $L^1(N)+L^p(N)$. Let 
$$
u_k(x,\rho)=C_1(n,m,s)\eta_k\rho^{2s}[\psi_s\ast \varphi_{s,\rho}](x),
$$
with $\eta_k$ as in the proof of Theorem \ref{thm:tHihomo}. Then, $u_k(x,0)=\eta_k\psi_s(x)\to \psi_s(x)$ as $k\to \infty$.

Let us now take $u=u_k$ in the inequality \eqref{eq:ideo}. As $k\to \infty$ the left hand side converges to $\int_N(\widetilde{w}_s(x)-w_s(x))\psi_s(x)\,dx$. On the other hand, as $u_k(x,\rho)=\eta_kw(x,\rho)$ it follows that $\big| \N u_k(x,\rho)-\eta_k\N(x,\rho)\big|^2\to 0$ as $k\to \infty$. Consequently, the right hand side tends to $0$, proving $\widetilde{w}_s= w_s$, as $\psi_s(x)>0$.
\end{rem}
\subsubsection*{Acknowledgments.}

This research was partially supported through the program ``Research in Pairs'' by the Mathematisches Forschungsintitut Oberwolfach in 2017. The first author wishes to thank \'Oscar Ciaurri and Pablo R. Stinga for
introducing her to Hardy inequalities and for many fruitful
discussions.

\end{document}